\documentclass[]{article}
\usepackage{amsmath, amsfonts, fullpage, xspace, color, amsthm, hyperref, mathtools, caption}
\usepackage{tikz, pgfplots, tikz-3dplot}
\usetikzlibrary{calc, fit, arrows, positioning, arrows.meta}

\newtheorem{lemma}{Lemma}[section]

\newtheorem{cor}[lemma]{Corollary}
\newtheorem{theorem}[lemma]{Theorem}
\newtheorem{prop}[lemma]{Proposition}

\newtheorem{hyp}{Hypothesis}

\theoremstyle{definition}
\newtheorem{defn}[lemma]{Definition}

\newtheorem*{example*}{Example}

\numberwithin{equation}{section}

\newcommand{\re}{\mbox{Re}\,\xspace}

\title{Equivariant Pyragas control of discrete waves}

\date{}
\author{Babette de Wolff\footnote{Vrije Universiteit Amsterdam, Department of Mathematics, \href{mailto:b.wolff@vu.nl}{b.wolff@vu.nl}}}

\begin{document}
\maketitle

\begin{abstract}
Equivariant Pyragas control is a delayed feedback method that aims to stabilize spatio-temporal patterns in systems with symmetries. In this article, we apply equivariant Pyragas control to discrete waves, which are periodic solutions %whose set of spatio-temporal relations forma discrete group. 
that have 
a finite number of spatio-temporal symmetries. We prove sufficient conditions under which a discrete wave can be stabilized via equivariant Pyragas control. The result is applicable to a broad class of discrete waves, including discrete waves that are %not close to a bifurcation point.
far away from a bifurcation point.  
%is formulated in terms of eigenvalue properties of the uncontrolled system, and are applicable to a broad class of discrete waves. 
Key ingredients of the proof are %an equivariant adaptation of Floquet theory, 
an adaptation of Floquet theory to systems with symmetries, 
and the use of characteristic matrix functions to reduce the infinite dimensional eigenvalue problem to a one dimensional zero finding problem. \\


\end{abstract}

\section{Introduction}
In \cite{Pyragas92}, Kestitutis Pyragas introduced a delayed feedback control scheme (now known as `Pyragas control') that aims to stabilize periodic motion. Pyragas considers a system without feedback that is described by an ordinary differential equation (ODE)
\begin{equation}
\dot{x}(t) = f(x(t)), \qquad t \geq 0 \label{eq:ode}
\end{equation}
with $f: \mathbb{R}^N \to \mathbb{R}^N$. Pyragas then introduces a feedback term that measures the difference between the current state and the state time $p$ ago, and then feeds this difference (multiplied by a matrix) back into the system. Concretely, the system with feedback control becomes
\begin{equation}\label{eq:pyragas}
\dot{x}(t) = f(x(t)) + B \left[x(t)-x(t-p) \right]
\end{equation}
with \emph{gain matrix} $B \in \mathbb{R}^{N \times N}$. For a periodic solution with period $p$, the difference between the current state and the state time $p$ ago is zero. Hence a $p$-periodic solution of the original system \eqref{eq:ode} is also a solution of the feedback system \eqref{eq:pyragas}. However, the global dynamics of the systems \eqref{eq:ode} and \eqref{eq:pyragas} are radically different, 
and we can try to choose the matrix $B \in \mathbb{R}^{N \times N}$ in such a way that an unstable $p$-periodic solution of \eqref{eq:ode} is a stable solution of \eqref{eq:pyragas}. 
%$x_\ast$ is a stable solution of \eqref{eq:pyragas}. 

If the original system \eqref{eq:ode} has built-in symmetries, 
its periodic solutions can satisfy additional spatio-temporal relations. In this case, one can adapt the Pyragas control scheme so that it vanishes on solutions with a prescriped spatio-temporal relation. 
As in \cite{Fiedler10}, we write the feedback system as
\begin{equation}\label{eq:pyr equiv}
\dot{x}(t) = f(x(t)) + B \left[x(t) - hx(t-\theta_h) \right] 
\end{equation}
with time delay $\theta_h > 0$ and $h \in \mathbb{R}^{N \times N}$ a linear, spatial transformation. The control term now feeds back the difference between the current state and a spatio-temporal transformation of the state, and vanishes on \emph{spatio-temporal patterns} of the form $hx(t) = x(t+\theta_h)$. The feedback scheme \eqref{eq:pyr equiv} has the advantage that it is able to select a prescribed spatio-temporal pattern amongst a family of periodic solutions with the same period, and in such situations can indeed be more succesful in stabilizing a specific pattern than Pyragas control \cite{Schneider16, Postelthwaite13}. Since symmetries of the uncontrolled system \eqref{eq:ode} are often described in terms of equivariance relations, we refer to the control scheme \eqref{eq:pyr equiv} as \textbf{equivariant Pyragas control}. 

Implementation of Pyragas control requires knowledge of the period of the targeted solution, but uses no additional information on the original system \eqref{eq:ode}. 
% the model \eqref{eq:ode intro}.
This `model-independence' makes Pyragas control widely applicable; for example in semiconductor lasers \cite{Schiroka06, Schikora11}, 
$CO_2$-lasers \cite{Bielawski} 
%helicopters carrying heavy loads \cite{Omar} 
and enzymatic reactions \cite{Lekebusch}; the paper \cite{Pyragas92} has currently (October 2022) more than 3100 citations. 
The %`spatio-temporal' feedback scheme
equivariant control scheme \eqref{eq:pyr equiv} has recently been experimentally implemented in networks of chemical oscillators \cite{Schneider21}. 

While Pyragas control has many experimental realizations, proving mathematically rigorous results on Pyragas control is challenging. 
This is mainly because, from a mathematical perspective, 
the controlled systems \eqref{eq:pyragas} and \eqref{eq:pyr equiv} are delay differential equations (DDE) that generate infinite dimensional dynamical systems. 
In order to associate to \eqref{eq:pyragas} (resp. \eqref{eq:pyr equiv}) a well-posed initial value problem, we have to provide a function on the interval $[-p, 0]$ (resp. $[-\theta_h, 0]$) as initial condition. 
So the state space is a function space (that has to be specified more precisely) and the associated dynamical system is infinite dimensional. Although the abstract theory of DDE is well developed \cite{HaleVL93, Diekmann95}, the infinite dimensional nature of DDE is still demanding when we want to perform an explicit stability analysis in concrete examples. 

\medskip

This article is concerned with equivariant Pyragas control of \emph{discrete waves}, which are periodic solutions that have a finite number of spatio-temporal symmetries. 
The main result of this article, Theorem \ref{thm:main result}, provides sufficient conditions under which a discrete wave can be stabilized via equivariant Pyragas control. The sufficient conditions are formulated in terms of eigenvalue properties
%in relatively concrete properties 
of the uncontrolled system, and the result is applicable to a broad class of discrete waves. 

The results in this article in particular apply to periodic orbits that are not close to a bifurcation point and that are `genuine' periodic orbits, i.e. they cannot be transformed to a ring of equilibria of an autonomous system. This is significant, since in the literature so far,
most analytical results on succesful stabilization by (equivariant) Pyragas control either concern periodic orbits that bifurcate from an equilibrium \cite{Hooton19, Fiedler20, VL17, Hooton17} or concern \emph{rotating waves}, i.e. periodic orbits that can be transformed to equilibria of autonomous systems \cite{Purewal14,Fiedler08,Schneider16,Fiedler10, Schneider22}.  
In both these cases, 
the stability analysis simplifies, because we can determine the stability of the periodic orbit by determining the stability of an  equilibrium in an autonomous system. 
These simplifications cannot be made in the setting considered in Theorem \ref{thm:main result} and consequently the stability problem becomes more involved. 

The stability analysis we perform here is based on a combination of equivariant Floquet theory with the theory of characteristic matrix functions. In systems without symmetry, we determine the stability of a $p$-periodic solution using the monodromy operator, which involves solving the linearized equation over a time step $p$. We show that in equivariant settings, we can work with the \emph{twisted monodromy operator}, which involves solving the linearized equation over only a fraction of the period. 
We then prove that the twisted monodromy operator has a \emph{characteristic matrix function}, a concept that was recently introduced in \cite{KaashoekVL22}. A characteristic matrix function captures the spectrum of a bounded linear operator in a matrix valued function, and using this concept we rigorously prove that the eigenvalues of the twisted monodromy operator are zeroes of a scalar valued function. 
This translates the infinite dimensional eigenvalue problem of the twisted monodromy operator to a one dimensional zero finding problem, 
which means a significant dimension reduction.
As the final step, we analyze the scalar valued function and %prove this article's main result
prove sufficient conditions under which the control scheme \eqref{eq:pyr equiv} succesfully stabilizes a discrete wave. 

\medskip

This article is structured as follows. In Section \ref{sec:setting}, 
we state the main result (Theorem \ref{thm:main result}) in mathematically precise form, after having introduced the necessary terminology. In Section \ref{sec:equivariance}, we describe the symmetry relations of the controlled system \eqref{eq:pyr equiv}; we then prove that the stability of discrete wave solutions of \eqref{eq:pyr equiv} is determined by the spectrum of the twisted monodromy operator. 
In Section \ref{sec:cm}, we introduce the concept of a characteristic matrix function; and show that the eigenvalues of the twisted monodromy operator can be computed as zeroes of a scalar-valued function.
We analyze this scalar-valued function in Section \ref{sec:eigenvalues} and subsequently prove Theorem \ref{thm:main result}. 

\subsection*{Acknowledgements}
The contents of this article are based upon contents of the authors doctoral thesis \cite{proefschrift}, written at the Freie Universit{\"a}t Berlin under the supervision of Bernold Fiedler. The author is grateful to Bernold Fiedler and Sjoerd Verduyn Lunel for useful discussions and encouragment; and to Jia-Yuan Dai, Bob Rink and Isabelle Schneider for comments on earlier versions. 

\section{Setting and statement of the main result} \label{sec:setting}
Throughout the rest of this article, we assume that the ODE \eqref{eq:ode} is equivariant with respect to a compact Lie group $\Gamma$. This means that there exists a group homomorphism
\[ \rho: \Gamma \to GL(N, \mathbb{R}) \]
(called a \textbf{representation} of $\Gamma$) such that
\begin{equation} \label{eq:representation}
f(\rho(\gamma)x) = \rho(\gamma) f(x)
\end{equation}
for all $x \in \mathbb{R}^N$ and all $\gamma \in \Gamma$. If now $x(t)$ is a solution of \eqref{eq:ode} and $\gamma$ is an element of $\Gamma$, then $\rho(\gamma)x(t)$ is a solution of \eqref{eq:ode} as well. So the group $\Gamma$ (or rather the group $\{ \rho(\gamma) \mid \gamma \in \Gamma \}$) is indeed a group of symmetries of the solutions of \eqref{eq:ode}. In many examples, symmetries of an ODE can be effectively described using compact Lie groups; see for example the monographs \cite{Golubitsky88,Golubitsky02}.

A compact Lie group $\Gamma$ always has a orthogonal representation, i.e. there always exists a group homomorphism
%A compact Lie group $\Gamma$ can always by represented by a subgroup of the orthogonal group, i.e. there always exists a group homomorphism 
\[ \rho: \Gamma \to O(N), \]
cf. \cite[p. 31]{Golubitsky88}. In the rest of this article, we directly view $\Gamma$ as a subgroup of the orthogonal group $O(N)$, instead of viewing it as an abstract compact Lie group with an orthogonal representation. Consequently we also supress the representation in the notation, e.g. we now write the equivariance condition \eqref{eq:representation} as
\[ f(\gamma x) = \gamma f(x)\]
with $\gamma \in \Gamma \subseteq O(N)$ and $x \in \mathbb{R}^N$. 

\medskip

The symmetries of the ODE \eqref{eq:ode} naturally induce two symmetry groups on a given periodic solution. Suppose that $x_\ast$ is a periodic solution of \eqref{eq:ode} with minimal period $p > 0$; denote by $\mathcal{O} = \{x_\ast(t) \mid t \in \mathbb{R} \}$ its orbit. Then the group
%The symmetry group $\Gamma$ of the ODE \eqref{eq:ode} gives rise to two symmetry groups on the periodic solution $x_\ast$. Indeed, consider the group
\[ K : = \{ \gamma \in \Gamma \mid \gamma x_\ast(0) = x_\ast(0) \} \]
leaves the initial condition $x_\ast(0)$ invariant, and the group 
\[ H := \{ \gamma \in \Gamma \mid \gamma \mathcal{O} = \mathcal{O} \} \]
leaves the orbit $\mathcal{O}$ invariant; cf. \cite{Fiedler88}. If $k \in K$, then $kx_\ast(t)$ and $x_\ast(t)$ are two solutions of \eqref{eq:ode} with the same intial condition, and hence
\[ k x_\ast(t) = x_\ast(t) \]
for all $t \in \mathbb{R}$. So elements of the group $K$ leave the orbit of $x_\ast$ fixed pointwise and hence we refer to the group $K$ as the group of \textbf{spatial symmetries} of $x_\ast$. For any an element $h \in H$, there exists a time-shift $\theta_h \in [0,p)$ such that $hx_\ast(0) = x_\ast(\theta_h)$. But then $hx_\ast(t)$ and $x_\ast(t+\theta_h)$ are both solutions of \eqref{eq:ode} with the same initial condition, and hence
\begin{equation} \label{eq:pattern}
hx_\ast(t) = x_\ast(t+\theta_h)
\end{equation}
for all $t \in \mathbb{R}$. 
So every element of $H$ induces a spatio-temporal relation of the form \eqref{eq:pattern} on $x_\ast$, and hence we refer to the group $H$ as the group of \textbf{spatio-temporal symmetries} of $x_\ast$.

If $h, g \in H$ are two spatio-temporal symmetries of $x_\ast$, then $ g  h x_\ast(t) = x_\ast(t+\theta_{h} + \theta_g )$
and hence $\theta_{hg} = \theta_h + \theta_g \mod p$. Thus the map 
\begin{align*}
H &\to S^1 \cong  \mathbb{R}/ p \, \mathbb{Z} \\
 h &\mapsto \theta_h
\end{align*}
is a group homomorphism. Since the group $K$ is exactly the kernel of the map $H \ni h \mapsto \theta_h$, it is in particular a normal subgroup of $H$, and the quotient group $H/K$ is a subgroup of $S^1$. This implies  that
\begin{align*}
\begin{cases}
H/K &\cong \mathbb{Z}_n \qquad \mbox{for some } n \in \mathbb{N}, \mbox{ or} \\
H/K &\cong S^1,
\end{cases}
\end{align*} 
where $\mathbb{Z}_n$ denotes the cyclic group of order $n$. If $H/K \simeq S^1$, the periodic solution $x_\ast$ is called a \textbf{rotating wave}; if $H/K \simeq \mathbb{Z}_n$ the periodic solution $x_\ast$ is often called a \textbf{discrete wave}, cf. \cite{Fiedler88}. For discrete waves, the time-shift $\theta_h$ associated to spatio-temporal symmetry $h \in H$ is always rationally related to the minimal period $p$ of the orbit. Indeed, if $h \in H$ and $H/K \simeq \mathbb{Z}_n$, then necessarily $h^n \in K$. 
So $n \theta_h = 0 \mod p$ and therefore there exists an integer $m \in \{1, \ldots, n\}$ such that 
\begin{equation} \label{eq:rational relation}
 \theta_h  = \frac{m}{n} p, 
\end{equation}
i.e. $\theta_h$ and $p$ are rationally related.

Throughout the rest of this article, we focus on the stabilization of discrete waves. For future reference, we collect the relevant assumptions on the ODE \eqref{eq:ode} in a seperate hypothesis. 

\begin{hyp} \label{hyp:theorem} \hfill
\begin{enumerate}
\item $f: \mathbb{R}^N \to \mathbb{R}^N$ is a $C^2$ function;
\item system \eqref{eq:ode} has a periodic solution $x_\ast$ with \emph{minimal} period $p > 0$;
\item system \eqref{eq:ode} is equivariant with respect to a compact symmetry group $\Gamma \subseteq O(N)$, i.e. 
\begin{equation} \label{eq:equivariance}
f(\gamma x) = \gamma f(x) \qquad \mbox{for all } x \in \mathbb{R}^N \mbox{ and } \gamma \in \Gamma.
\end{equation}
\item The periodic solution $x_\ast$ is a discrete wave, i.e. $H/K \simeq \mathbb{Z}_n$ for some $n \in \mathbb{N}$. 
\end{enumerate}
\end{hyp}

\medskip

\noindent To determine whether the $p$-periodic discrete wave $x_\ast$ is a stable solution of \eqref{eq:ode}, we consider the linearized equation 
\begin{equation} \label{eq:linear ode}
\dot{y}(t) = f'(x_\ast(t)) y(t),
\end{equation}
which is non-autonomous and $p$-periodic in its time argument. We denote by $Y(t) \in \mathbb{R}^{N \times N}$ the \textbf{fundamental solution} of \eqref{eq:linear ode} with $Y(0) = I$, i.e. $Y(t)$ is the matrix-valued solution of the initial value problem 
\[ \frac{d}{dt} Y(t) = f'(x_\ast(t))Y(t), \qquad Y(0) = I. \]
Floquet theory implies that the eigenvalues of the \textbf{monodromy operator} $Y(p) \in \mathbb{R}^{N \times N}$ determine whether $x_\ast$ is a stable solution of \eqref{eq:ode}. The equivariance assumption in Hypothesis \ref{hyp:theorem} allows us to refine Floquet theory for discrete waves. We do this in detail in Section \ref{sec:equivariance}; for now we just mention that in the stability analysis of discrete waves, the operator 
%at this point we only say that 
\begin{equation} \label{eq:twisted mon ode}
 Y_h: \mathbb{R}^N \to \mathbb{R}^N, \qquad Y_h  = h^{-1} Y(\theta_h) \quad \mbox{with } h \in H,
\end{equation}
plays an important role; we call the operator \eqref{eq:twisted mon ode} the \textbf{twisted monodromy operator (associated to h)}.
The twisted monodromy operator $Y_h$ always has an eigenvalue $1 \in \mathbb{C}$, which we call the \textbf{trivial eigenvalue}. This is because differentiating the relation $\dot{x}_\ast(t) = f(x_\ast(t))$ with respect to time implies that $\dot{x}_\ast(t)$ is a solution of the linearized equation \eqref{eq:linear ode}. So $Y(t)\dot{x}_\ast(0) = \dot{x}_\ast(t)$
and together with \eqref{eq:pattern} this implies that 
\[ h^{-1} Y(\theta_h)\dot{x}_\ast(0) = h ^{-1} x_\ast(\theta_h) = \dot{x}_\ast(0). \]
In Section \ref{sec:equivariance}, we additionally prove that if 
$Y_h$ has an eigenvalue $\left| \mu \right| > 1$, then $x_\ast$ is an unstable solution of \eqref{eq:ode}. 

\medskip

With these preparations, we are now ready to state this article's main result. 

\begin{theorem} \label{thm:main result}
Consider the ODE \eqref{eq:ode} satisfying Hypothesis \ref{hyp:theorem}. %Assume that the periodic orbit $x_\ast$ is a discrete wave. Moreover, 
Assume that the discrete wave $x_\ast$ has a spatio-temporal symmetry $h \in H$ such that the twisted monodromy operator $Y_h$ defined in \eqref{eq:twisted mon ode} has the following properties:
\begin{enumerate}
\item The eigenvalue $1 \in \sigma(Y_h)$ is algebraically simple and $Y_h$ has no other eigenvalues on the unit circle;
\item If $\mu \in \sigma(Y_h)$ and $\left| \mu \right| > 1$, then 
\begin{equation}
- e^2 < \mu < -1. 
\end{equation}
\end{enumerate}
Then there exists an open interval $I \subseteq (-\infty, 0)$ such that for $b \in I$, $x_\ast$ is a stable solution of 
\begin{equation} \label{eq:control thm}
\dot{x}(t) = f(x(t)) + b \left[x(t) - h x(t-\theta_h)\right]. 
\end{equation}
\end{theorem}

\medskip
\noindent
Theorem \ref{thm:main result} addresses equivariant Pyragas control %/a control scheme 
with \textbf{scalar control gain}, i.e. in the control term \[b \left[x(t) - h x(t-\theta_h) \right]\] the factor $b$ is a real number rather than a matrix. 
The fact that Theorem \ref{thm:main result} achieves stabilization with a scalar control gain is remarkable %noteworthy
%in the light of the following result on Pyragas control:
since non-equivariant Pyragas control with a scalar control gain fails to stabilize a rather large class of periodic solutions. Indeed, if $x_\ast$ is a $p$-periodic solution of the ODE \eqref{eq:ode}, and the monodromy operator $Y(p)$ has at least one real eigenvalue $\mu > 1$, then $x_\ast$ is an unstable solution of the controlled system 
\begin{equation} \label{eq:scalar pyragas}
\dot{x}(t) = f(x(t)) + b \left[x(t) - x(t-p) \right]
\end{equation}
for every choice of $b \in \mathbb{R}$ \cite{deWolff21}. 
%So Pyragas control with a scalar control gain fails to stabilize for a rather large class of systems.  
Although Theorem \ref{thm:main result} makes assumptions on the eigenvalues of the twisted monodromy operator $Y_h$, it does not make assumptions on the eigenvalues of the monodromy operator $Y(p)$. 
Given a periodic orbit $x_\ast$, it is possible that its monodromy operator $Y(p)$ has an eigenvalue $\mu > 1$, while its twisted monodromy operator $Y_h$ satisfies the assumptions of Theorem \ref{thm:main result}. 
In this situation, the equivariant control scheme \eqref{eq:control thm} overcomes a limitation to Pyragas control in the sense that the periodic solution can be stabilized using the control \eqref{eq:control thm} but is always an unstable solution of \eqref{eq:scalar pyragas}.  
 
A concrete example of this situation occurs in the Lorenz equation
\begin{equation} \label{eq:lorenz}
\begin{aligned}
\begin{cases}
\dot{x}_1 &= - \sigma x_1 + \sigma x_2, \\
\dot{x}_2 &= -x_1 x_3 + \lambda x_1 - x_2, \\
\dot{x}_3 &= x_1 x_2 - \epsilon x_3,
\end{cases}
\end{aligned}
\end{equation}
with $x_1, x_2, x_3 \in \mathbb{R}$ and 
with parameters $\sigma, \epsilon, \lambda \in \mathbb{R}$.  System \eqref{eq:lorenz} is symmetric with respect to the group $\mathbb{Z}_2 = \{e, 
\gamma\}$, where $e$ is the identity element of the group and we represent $\gamma$ on $\mathbb{R}^3$ as
\[ (x_1, x_2, x_3) \mapsto (-x_1, - x_2, x_3).\]
In \cite{Wulff06}, Wulff and Schebes numerically show that for parameter values $\sigma = 10, \ \epsilon = 8/3$ and $\lambda = 312$ system \eqref{eq:lorenz} has a periodic solution with $H = \mathbb{Z}_2$ and $K = \{e \}$. 
The authors continue this periodic orbit with respect to the parameter $\lambda$, while keeping the parameters $\sigma$ and $\epsilon$ fixed. They find that for $\lambda \approx 312.97$, the periodic orbit undergoes a \emph{flip-pitchfork bifurcation} (a term introduced in \cite{Fiedler88}), which means that there exist parameter values close to $\lambda \approx 312.97$ where the twisted monodromy operator $Y_h$ has an eigenvalue $\mu < -1$ and satisfies the assumptions of Theorem \ref{thm:main result}, whereas the monodromy operator $Y(p)$ has an eigenvalue larger than $1$. So in this parameter regime the periodic orbit can be stabilized using the equivariant control scheme \eqref{eq:control thm}, although it cannot be stabilized using Pyragas control of the form \eqref{eq:scalar pyragas}. 

\section{Equivariant Floquet theory} \label{sec:equivariance}
Throughout this section, we fix a spatio-temporal symmetry $h \in H$ together with a scalar control gain $b \in \mathbb{R}$ and consider the controlled system \eqref{eq:control thm}. We first %comment on/
identify the symmetries of system \eqref{eq:control thm}, and then develop an equivariant Floquet theory for its linearization. 

\medskip

We can write the controlled system \eqref{eq:control thm} as
\begin{subequations}
\begin{equation} \label{eq:dde g}
\dot{x}(t) = g(x(t), x(t-\theta_h))
\end{equation}
with $g: \mathbb{R}^N \times \mathbb{R}^N \to \mathbb{R}^N$ defined as 
\begin{equation} \label{eq:g}
g(x, y) = f(x) + b \left[x - h y \right]. 
\end{equation}
\end{subequations}
The function $g$ is equivariant with respect to the group generated by $h$, in the sense that 
\[ g(h^j x, h^j y) = h^j g(x, y)  \]
holds for all $x, y \in \mathbb{R}^N$ and for all $j \in \mathbb{N} \cup \{0 \}$. 
If now $x(t)$ is a solution of the DDE defined by \eqref{eq:dde g}--\eqref{eq:g}, then 
\begin{align*}
\frac{d}{dt} \left( h^j x(t) \right) &= h^j g(x(t), x(t-\theta_h)) \\
& = g(h^j x(t), h^j x(t-\theta_h))
\end{align*}
and hence $h^j x(t)$ is a solution of \eqref{eq:dde g}--\eqref{eq:g} as well. So the DDE defined by \eqref{eq:dde g}--\eqref{eq:g} (or, equivalently, the DDE \eqref{eq:control thm}) is equivariant with respect to the group generated by $h$ in the sense that elements of this group send solutions to solutions. 

In general, not every symmetry of the uncontrolled system \eqref{eq:ode} is also a symmetry of the controlled system \eqref{eq:dde g}--\eqref{eq:g}: if the symmetry group $\Gamma$ of the ODE \eqref{eq:ode} is not abelian, and $\gamma \in \Gamma$ is a group element such that $\gamma h \neq h \gamma$, then $\gamma$ is not a symmetry of \eqref{eq:dde g}--\eqref{eq:g}. So in general the symmetry group of \eqref{eq:dde g}--\eqref{eq:g} contains a subgroup of $\Gamma$ (namely the group generated by $h$) but it need not be the entire group $\Gamma$. 

\medskip

We next 
consider the linearized equation
\begin{subequations}
\begin{equation} \label{eq:linearized control}
\dot{y}(t) = f'(x_\ast(t)) y(t) + b \left[y(t) -  hy(t-\theta_h) \right];
\end{equation}
if we fix $s \in \mathbb{R}$ and supplement \eqref{eq:linearized control} with the initial condition
\begin{equation} \label{eq:initial condition}
y(s+t) = \varphi(t) \qquad \mbox{for } t \in [-\theta_h, 0] \mbox{ and } \varphi \in C\left([-\theta_h, 0], \mathbb{R}^N\right), 
\end{equation}
\end{subequations}
then the system \eqref{eq:linearized control}--\eqref{eq:initial condition} has a unique solution $y(t)$ for $t \geq s$ \cite[Chapter 12]{Diekmann95}. Given a time $t \geq s$, we define the \textbf{history segment} $y_t \in C \left([-\theta_h, 0], \mathbb{R}^N\right)$ at time $t$ as %of this solution as
$y_t(\vartheta) = y(t+\vartheta), \ \vartheta \in [-\theta_h, 0]$. %\[ y_t \in C \left([-\tau, 0], \mathbb{R}^N\right), \qquad y_t(\theta) = y(t+\theta).\]
We then associate to \eqref{eq:linearized control}--\eqref{eq:initial condition} a two-parameter system of operators 
\begin{equation} \label{eq:two parameter system}
U(t, s): C\left([-\theta_h, 0], \mathbb{R}^N\right) \to C\left([-\theta_h, 0], \mathbb{R}^N\right), \qquad t \geq s
\end{equation}
with the property that $y_t = U(t, s) \varphi$ is the solution of \eqref{eq:linearized control} with initial condition \eqref{eq:initial condition} at time $s$. We refer to 
%two-parameter family $\{U(t,s)\}_{t \geq s}$ 
\eqref{eq:two parameter system}
as the \textbf{family of solution operators}
%associated to \eqref{eq:linearized dde}. 
of \eqref{eq:linearized control}. Since the non-autonomous system \eqref{eq:linearized control} is $p$-periodic in its time argument, standard Floquet theory for DDE implies that 
\begin{equation} \label{eq:floquet}
U(t+p, s+p) = U(t, s) 
\end{equation}
for all $t \geq s$; see \cite[Chapter 13]{Diekmann95}. The next lemma 
shows that the symmetry relations on \eqref{eq:control thm} induce additional spatio-temporal relations on the family of solution operators $U(t, s)$.  

\begin{lemma} \label{lem: fundamental solution dde} 
Consider the ODE \eqref{eq:ode} satisfying Hypothesis \ref{hyp:theorem}. For a fixed spatio-temporal symmetry $h \in H$ and scalar control gain $b \in \mathbb{R}$,
%\[ \dot{x}(t) = f(x(t)) + b \left[x(t) - h x(t-\theta_h) \right]. \]
let $U(t, s), \ t \geq s$, be the family of solution operators associated to the linearized system \eqref{eq:linearized control};  let $n \in \mathbb{N}$ be such that $H/K \simeq \mathbb{Z}_n$. 
Then 
\begin{subequations}
\begin{equation}\label{eq: U h}
 h U(t, s) = U(t+\theta_h, s+\theta_h) h
\end{equation}
and 
\begin{equation} \label{eq: U k}
h^n U(t, s) = U(t, s) h^n
\end{equation}
\end{subequations}
for all $t \geq s$. 
\end{lemma}
\begin{proof} 
The equivariance relation \eqref{eq:equivariance} implies that 
\begin{equation} \label{eq:commute}
h f(x) = f(hx) \qquad \mbox{for all } x \in \mathbb{R}^N
\end{equation}
and differentiating \eqref{eq:commute} with respect to $x$ yields 
\[  h f'(x) = f'(h x) h \qquad  \mbox{for all } x \in \mathbb{R}^N. \]
Since $hx_\ast(t) = x_\ast(t+\theta_h)$, this implies that
\[ h f'(x_\ast(t)) = f'(hx_\ast(t)) h = f'(x_\ast(t+\theta_h)) h  \]
for all $t \in \mathbb{R}$. 
Now fix $s \in \mathbb{R}$ and $\varphi \in C \left([-\theta_h, 0], \mathbb{R}^N\right)$, and let $y(t)$ be the unique solution of the initial value problem
\begin{align*}
\begin{cases}
\dot{y}(t) &= f'(x_\ast(t)) y(t) + b \left[y(t) - y(t-\theta_h)\right], \qquad  t \geq s, \\
y(t) &= \varphi(t), \qquad   t \in [s-\theta_h, s]
\end{cases}
\end{align*}
so that $y_t = U(t, s) \varphi$. Then $h y(t)$ satisfies
\begin{align*}
\frac{d}{dt} (h y(t))& = h \dot{y}(t) \\ & = h f'(x_\ast(t)) y(t) + h b \left[y(t) - hy(t-\theta_h) \right] \\
 & = f'(x_\ast(t+\theta_h)) h y(t) + b \left[ hy(t) - h \left[h y(t-\theta_h) \right] \right]. 
\end{align*}
So $h y(t)$ is a solution of the initial value problem 
\begin{align*}
\begin{cases}
\dot{z}(t) &= f'(x_\ast(t+\theta_h)) z(t) + b \left[z(t) - hz(t-\theta_h)\right], \qquad  t \geq s \\
z(t) &= h \varphi(t), \qquad   t \in [s-\theta_h, s]. 
\end{cases}
\end{align*}
But then uniqueness of solutions implies that $hy_t = U(t+\theta_h, s+\theta_h) h \varphi$. So $h U(t, s) \varphi = U(t+\theta_h, s+\theta_h) h \varphi$ for all $\varphi \in C \left([-\theta_h, 0], \mathbb{R}^N\right)$, which proves \eqref{eq: U h}. 

\medskip 
Iteratively applying \eqref{eq: U h} gives
\begin{equation} \label{eq:iterates l}
h^n U(t, s) = U(t+n \theta_h, s+n \theta_h) h^n. 
\end{equation}
Since we have assumed that $x_\ast$ is a discrete wave with $H/K \simeq \mathbb{Z}_n$, there exists a $m \in \{1, \ldots, n \}$ such that $n \theta_h = mp$, cf. \eqref{eq:rational relation}. Substituting this into \eqref{eq:iterates l} gives that 
\[ h^n U(t, s) = U(t+mp, s+mp) h^n. \]
But iteratively applying \eqref{eq:floquet} implies that 
\[ U(t+mp, s+mp) = U(t, s) \]
and hence 
\[ h^n U(t, s) = U(t, s) h^n, \]
as claimed. 
\end{proof}

For $U(t, s), t \geq s$, the family of solution operators of \eqref{eq:linearized control}, and $h \in H$ the spatio-temporal symmetry used for control in \eqref{eq:control thm}, we define the \textbf{twisted monodromy operator (associated to $h$)} as 
\begin{equation} \label{eq:twisted mon op}
U_h := h^{-1} U(\theta_h, 0).
\end{equation}
Here we slightly abuse notation and view the matrix $h^{-1} \in \mathbb{R}^{N \times N}$ as an operator on the Banach space $C\left([-\theta_h, 0], \mathbb{R}^N\right)$ acting as $\left(h^{-1} \phi\right)(\vartheta) = h^{-1} \phi(\vartheta), \ \vartheta \in [-\theta_h, 0]$. The notion of the twisted monodromy operator is inspired by an equivariant version of the Poincar\'e map introduced for ODE in \cite[p. 55]{Fiedler88}. The difference with \cite{Fiedler88} is that we work with the flow of the linearized system, rather than with the Poincar\'e map, and work in the context of DDE, rather than ODE. 

The next lemma provides a relation between the \textbf{monodromy operator} $U(p, 0)$, which plays an important role in non-equivariant Floquet theory, and the twisted monodromy operator $U_h$.  

\begin{lemma} \label{lem:decomposition operators}
Consider the ODE \eqref{eq:ode} satisfying Hypothesis \ref{hyp:theorem}. For a fixed spatio-temporal symmetry $h \in H$ and scalar control gain $b \in \mathbb{R}$, let $U(t, s), \ t \geq s$, be the family of solution operators associated to the linearized system \eqref{eq:linearized control}. Let $n \in \mathbb{N}$ be such that $H/K \simeq \mathbb{Z}_n$ and 
let $m \in \{1, \ldots, n \}$ be such that $\theta_h = \frac{m}{n}p$. Moreover, let 
\[ U_h = h^{-1} U\left(\frac{m}{n}p, 0\right) \]
be the twisted monodromy operator. Then 
\[ U(p, 0)^m = h^n U_h^n. \]
\end{lemma}
\begin{proof}
Iteratively applying \eqref{eq:floquet} gives that 
\[U(jp, (j-1)p) = U(p, 0) \]
for $j \geq 1$. This implies that
\begin{align*}
U(mp, 0) &= U(mp, (m-1)p) U((m-1)p, (m-2)p) \ldots U(2p, p) U(p, 0) \\
& = \underbrace{U(p, 0) \ldots U(p, 0)}_{m \text{ times}}. 
\end{align*}
So
\begin{equation} \label{eq:Floquet power}
 U(mp, 0) = U(p, 0)^m.
\end{equation}
Next, iteratively applying \eqref{eq: U h} with $\theta_h = \frac{m}{n}$ gives that
\begin{align*}
U\left(\frac{jm}{n}p, \frac{(j-1)m}{n}p \right)  = h^{j-1} U\left(\frac{m}{n}p, 0 \right) h^{1-j}
\end{align*}
for $j \geq 1$. This implies that
\begin{align*}
U(mp, 0) &= U \left(mp, mp - \frac{m}{n}p \right) \ldots U \left(\frac{2m}{n}p, \frac{m}{n}p \right) U \left(\frac{m}{n}p, 0 \right) \\
&= \left[h^{n-1} U \left(\frac{m}{n}p, 0 \right) h^{1-n}\right] \left[ h^{n-2} U \left(\frac{m}{n}p, 0 \right) h^{2-n} \right] \ldots \left[ h U \left(\frac{m}{n}p, 0 \right) h^{-1} \right] U \left(\frac{m}{n}p, 0 \right)\\
& = h^{n-1} U \left(\frac{m}{n}p, 0 \right) h^{-1} U \left(\frac{m}{n}p, 0 \right) \ldots h^{-1} U \left(\frac{m}{n}p, 0 \right) \\
& = h^{n} h^{-1} U \left(\frac{m}{n}p, 0 \right) h^{-1} U \left(\frac{m}{n}p, 0 \right) \ldots h^{-1} U \left(\frac{m}{n}p, 0 \right) \\
& = h^n \left(h^{-1} U \left(\frac{m}{n}p, 0 \right)\right)^n
\end{align*}
so
\begin{equation} \label{eq:U power h}
 U(mp, 0) = h^n U_h^n.
\end{equation}
Combining this with \eqref{eq:Floquet power} yields that 
\[ U(p, 0)^m = h^n U_h^n, \]
as claimed. 
\end{proof}

The next proposition shows that the stability of discrete wave solutions is determined by the eigenvalues of the twisted monodromy operator. For the second statement of the proposition, the proof strategy is to first relate the spectrum of the twisted monodromy operator $U_h$ to the spectrum of the monodromy operator $U(p, 0)$. We then in turn relate the spectrum of the monodromy operator $U(p, 0)$ to the stability of the periodic orbit $x_\ast$ in \eqref{eq:control thm} by invoking stability theory for DDE (see \cite[Chapter 10]{HaleVL93} and \cite[Chapter XIV]{Diekmann95}). 

In order to discuss the spectrum of the twisted monodromy operator $U_h$, we first have to complexify the operator $U_h$ via a canonical procedure as detailed in, for example, \cite[Chapter 3]{Diekmann95}. However, we do not make the complexification explicit in notation, i.e. we write $U_h$ for both the real operator on the real Banach space $C\left([-\theta_h, 0], \mathbb{R}^N\right)$ and the complexified operator on the complex Banach space $C\left([-\theta_h, 0], \mathbb{C}^N\right)$. 
%We then invoke Floquet theory for DDE (see ... and ...) to relate the spectrum of $U(p, 0)$ -- and thus of $U_h$ -- to the stability of the periodic orbit $x_\ast$ in \eqref{eq:control thm}. 
%The proof strategy for the second statement of the proposition is to relate the spectrum of the twisted monodromy operator $U_h$ to the spectrum of the monodromy operator $U(p, 0)$. We then invoke Floquet theory for DDE, as developed in ... and ..., to decide whether $x_\ast$ is an (un)stable solution of \eqref{eq:control thm}. 
%{\color{blue}For the second statement of the Proposition is to relate the spectrum of the twisted monodromy operator $U_h$ to the spectrum of the monodromy operator $U(p, 0)$; we then invoke standard results on Floquet theory for DDE [...] to make statements on the stability of the periodic orbit.}

\begin{prop} \label{prop:spectral stability}
Consider the ODE \eqref{eq:ode} satisfying Hypothesis \ref{hyp:theorem}. For a fixed spatio-temporal symmetry $h \in H$ and scalar control gain $b \in \mathbb{R}$, consider the controlled system \eqref{eq:control thm}. Let $U(t, s), \ t \geq s$, be the family of solution operators associated to the linearized system \eqref{eq:linearized control} and 
let
\[ U_h = h^{-1} U(\theta_h, 0) \]
be the twisted monodromy operator. Then the following statements hold: 
\begin{enumerate}
\item The non-zero spectrum of $U_h$ consists of isolated eigenvalues of finite algebraic multiplicity. 
\item If $U_h$ has an eigenvalue strictly outside the unit circle, then $x_\ast$ is an unstable solution of \eqref{eq:control thm}. If the trivial eigenvalue $1 \in \sigma_{pt}(U_h)$ is algebraically simple and all other eigenvalues of $U_h$ lie strictly inside the unit circle, then $x_\ast$ is a stable solution of \eqref{eq:control thm}. 
\end{enumerate}
\end{prop}
\begin{proof}
We divide the proof into three steps: 

\medskip

\noindent \textbf{\textsc{Step 1:}} %We first prove that the non-zero spectrum of $U_h$ consists of isolated eigenvalues of finite algebraic multiplicity. 
%We first prove that the non-zero spectrum of both the twisted monodromy operator $U_h$ and the monodromy operator $U(p, 0)$ consists of isolated eigenvalues of finite algebraic multiplicity.  
We start by proving the first statement of the proposition. 

For $s = 0$ and $t \in [0, \theta_h]$, the initial value problem \eqref{eq:linearized control}--\eqref{eq:initial condition} becomes
\begin{equation}\label{eq:ivp reduced}
\dot{y}(t) = f'(x_\ast(t)) y(t) + b \left[ y(t) - h \varphi(t-\theta_h)) \right], \qquad y(0) = \varphi(0). 
\end{equation}
Let $X(t)$ be the fundamental solution of the ODE $\dot{y}(t) = f'(x_\ast(t)) y(t) + b y(t)$, i.e. $X(t)$ is the matrix-valued solution to the initial value problem 
\begin{align*}
\frac{d}{dt} X(t) = f'(x_\ast(t)) X(t) + b X(t), \qquad X(0) = I. 
\end{align*}
Then we can solve \eqref{eq:ivp reduced} by variation of constants as 
\[ y(t) = X(t) \varphi(0) - \int_0^t X(t)  X(\zeta)^{-1} b h \varphi(\zeta-\theta_h) d\zeta. \]
Therefore $\left(U_h \varphi\right)(\vartheta) := h^{-1} y(\theta_h + \vartheta)$ with $\vartheta \in [-\theta_h, 0]$ is given by 
\begin{align*}
\left(U_h \varphi\right)(\vartheta) = X(\theta_h + \vartheta) \varphi(0) - \int_0^{t + \theta_h +\vartheta} X(t + \theta_h +\vartheta) X(\zeta)^{-1}b h \varphi(\zeta-\theta_h) d\zeta. 
\end{align*}
The operator 
\[R: C\left([-\theta_h, 0], \mathbb{R}^N\right) \to C\left([-\theta_h, 0], \mathbb{R}^N\right)\] 
defined by 
\[ (R \varphi)(\vartheta) = X(\theta_h + \vartheta) \varphi(0), \qquad \vartheta \in [-\theta_h, 0] \]
has finite dimensional range and is hence compact. Moreover, the operator 
\[V:  C\left([-\theta_h, 0], \mathbb{R}^N\right) \to  C\left([-\theta_h, 0], \mathbb{R}^N\right)\] 
defined by 
\[ (V \varphi)(\vartheta)  = - \int_0^{t + \theta_h +\vartheta} X(t + \theta_h +\vartheta) X(\zeta)^{-1}b h \varphi(\zeta-\theta_h) d\zeta, \qquad \vartheta \in [-\theta_h, 0] \]
is compact by the Arzel{\`a}-Ascoli theorem. Hence the operator $U_h = V + R$ is compact as well, and its non-zero spectrum consists of isolated eigenvalues of finite algebraic multiplicity, as claimed. 

%By Lemma \ref{lem:decomposition operators}, there exist $n \in \mathbb{N}$ and $m \in \{1, \ldots, m\}$ such that 
%\[ U(p, 0)^m = h^n U_h^n. \]
%Since the operator $U_h$ is compact, the operator is $U(p, 0)^m$ is compact as well, and the non-zero spectrum of $U(p, 0)^m$ consists of isolated eigenvalues of finite algebraic multiplicity. But then the spectrum of $U(p, 0)$ consists of isolated eigenvalues of finite algebraic multiplcity as well {\color{blue} (cf. ... )}. 

%The operator on $C\left([-\theta_h, 0], \mathbb{R}^N\right)$ defined by 
%\[ \varphi \mapsto X(\theta_h + . ) \varphi(0) \]
%has finite dimensional range and is hence compact. 
\medskip
\noindent \textbf{\textsc{Step 2:}} We next decompose the Banach space $X = C\left([-\theta_h, 0], \mathbb{C}^N\right)$ into subspaces which are invariant under both the twisted monodromy operator $U_h$ and the monodromy operator $U(p, 0)$. We then use this decomposition to relate the spectrum of twisted monodromy opreator $U_h$ to the spectrum of the monodromy operator $U(p, 0)$. 

By assumption, $h^n \in \mathbb{R}^{N \times N}$ is an orthogonal matrix and hence it is diagonalizable. So if $\lambda_1, \ldots, \lambda_d$ are the distinct eigenvalues of $h^n$, then we can decompose $\mathbb{C}^N$ as
\[ \mathbb{C}^N = Y_1 \oplus \ldots \oplus Y_d \qquad \mbox{with} \qquad Y_i = \ker\left(\lambda_i I - h^n\right), \ 1 \leq i \leq d.\]
Next we decompose the Banach space $X = C\left([-\theta_h, 0], \mathbb{C}^N\right)$ as
\[ X = X_1 \oplus \ldots \oplus X_d \qquad \mbox{with} \qquad X_i = C \left([-\theta_h, 0], Y_i\right), \ 1 \leq i \leq d. \]
Now fix $1 \leq j \leq d$ and let $\varphi \in X_j$, i.e. $\lambda_j \varphi- h^n \varphi = 0$. Then equality \eqref{eq: U k} implies that
\begin{align*}
\lambda_j U_h \varphi - h^n U_h \varphi = U_h (\lambda_j \varphi- h^n \varphi ) = 0 . 
\end{align*}
So $U_h \varphi \in X_j$ and hence the space $X_j$ is invariant under the twisted monodromy operator $U_h$. 
Similarly, if $\varphi \in X_j$, then 
\begin{align*}
\lambda_j U(p, 0) \varphi - h^n U(p, 0) \varphi = U(p, 0) (\lambda_j \varphi - h^n \varphi ) = 0. 
\end{align*}
So $U(p, 0) \varphi \in X_j$ and hence the space $X_j$ is invariant under the monodromy operator $U(p, 0)$ as well. 

%Since both the monodromy operator $U(p, 0)$ and the twisted monodromy operator $U_h$ leave the spaces $X_j, \ 1 \leq j \leq d$ invariant, we can write 
%Since the twisted monodromy operator $U_h$ leaves the spaces $X_j, \ 1 \leq j \leq d$ invariant, we can write $U_h$ as 
%As a consequence, we can decompose the twisted monodromy operator $U_h$ as 
%\begin{align*}
% U(p, 0)^m &= \left. U(p, 0)^m \right|_{X_1} + \ldots + \left. U(p, 0)^m \right|_{X_d} \\
% U_h &= \left. U_h \right|_{X_1} + \ldots + \left. U_h \right|_{X_d}.
%\end{align*}
%\[ U_h = \left. U_h \right|_{X_1} + \ldots + \left. U_h \right|_{X_d}. \]
%Similarly, since the monodromy operator $U(p, 0)$ leaves the spaces $X_j, \ 1 \leq j \leq d$ invariant, we can write $U(p, 0)$ as
%Therefore we can decompose $U(p, 0)^m$ as
%\begin{subequations}
%\begin{equation}\label{eq:U decomposition space}
% U(p, 0)^m = \left. U(p, 0)^m \right|_{X_1} + \ldots + \left. U(p, 0)^m \right|_{X_d}. 
%\end{equation}

By Lemma \ref{lem:decomposition operators}, there exists a $n \in \mathbb{N}$ and a $m \in \{1, \ldots, n \}$ such that $U(p, 0)^m = h^n U_h^n$. But since both the operator $U(p, 0)$ and the operator $U_h$ leave the spaces $X_i, \ 1 \leq i \leq d$ invariant, it holds that 
\[  \left. U(p, 0)^m \right|_{X_j} = \left. \Big (h^n U_h^n \Big) \right|_{X_j}. \]
But  $h^n \varphi = \lambda_j \varphi$ for $\varphi \in X_j$, and hence 
\begin{equation} \label{eq:lambdaj}
\left. U(p, 0)^m \right|_{X_j} = \lambda_j \left( \left. U_h \right|_{X_j} \right)^n
\end{equation}
for all $1 \leq j \leq d$. By the first statement of the lemma, all non-zero spectrum of $ \left. U_h \right|_{X_j}$ consists of isolated eigenvalues of finite algebraic multiplicity. Therefore equality \eqref{eq:lambdaj} implies that also all the non-zero spectrum of $\left. U(p, 0)^m\right|_{X_j}$ consists of isolated eigenvalues of finite algebraic multiplicity, and moreover it holds that 
\begin{equation} \label{eq:ev decomposition dde}
\sigma_{pt} \left(\left. U(p, 0)^m \right|_{X_j}\right) = \lambda_j \sigma_{pt} \left( \left( \left. U_h \right|_{X_j} \right)^n \right) = \lambda_j \sigma_{pt} \left(  \left. U_h \right|_{X_j} \right)^n \qquad \mbox{for all } 1 \leq j \leq d,
\end{equation}
where all equalities count algebraic multiplicities.
%Since $h^n \phi = \lambda_j \phi$ for $\phi \in X_j$, it holds that 
%\[ \left. U(p, 0)^m \right|_{X_j} = \left. \Big (h^n U_h^n \Big) \right|_{X_j} = \lambda_j \left( \left. U_h \right|_{X_i} \right)^n. \]
%By the first statement of the lemma, all non-zero spectrum of $ \left. U_h \right|_{X_i}$ consists of isolated eigenvalues of finite algebraic multiplicity. Therefore, equality ... implies that 

\medskip
\noindent \textbf{\textsc{Step 3:}} We now prove the second statement of the proposition. First suppose that twisted monodromy operator $U_h$ has an eigenvalue $\mu$ with $\left| \mu \right| > 1$. Since 
\[ U_h = \left. U_h \right|_{X_1} + \ldots + \left. U_h \right|_{X_d}, \]
there exists a $1 \leq j \leq d$ such that $\left. U_h \right|_{X_j}$ has an eigenvalue $\mu$. Equality \eqref{eq:ev decomposition dde} then implies that $\left. U(p, 0)^m \right|_{X_j}$ has an eigenvalue $\lambda_j \mu^n$. Since the matrix $h^n$ is orthogonal, its eigenvalue $\lambda_j$ has norm one; hence $\left| \lambda_j \mu^n \right| = \left| \mu \right|^n > 1$ and $U(p,0)^m$ has an eigenvalue strictly outside the unit circle. This implies that 
%since $\left| \lambda_j \right| = 1$, this eigenvalue lies outside the unit circle. So $U(p, 0)^m$ has an eigenvalue strictly outside the unit circle, and therefore 
$U(p, 0)$ has an eigenvalue strictly outside the unit circle as well. Stability theory for DDE (see \cite[Chapter 10]{HaleVL93} and \cite[Chapter XIV]{Diekmann95}) then implies that $x_\ast$ is an unstable solution of \eqref{eq:control thm}, as claimed. 

Vice versa, assume that the eigenvalue $1 \in \sigma_{pt}(U_h)$ is algebraically simple and all other eigenvalues of $U_h$ lie inside the unit circle. Since 
\[ U(p, 0)^m = \left. U(p, 0)^m \right|_{X_1} + \ldots + \left. U(p, 0)^m \right|_{X_d}, \]
equality \eqref{eq:ev decomposition dde} implies that $U(p, 0)^m$ has an algebraically simple eigenvalue; all other spectrum of $U(p, 0)^m$ lies strictly inside the unit circle. But then also $U(p, 0)$ has an an algebraically simple eigenvalue and all other spectrum of $U(p, 0)$ lies strictly inside the unit circle. Stability theory for DDE then implies that $x_\ast$ is a stable solution of \eqref{eq:control thm}. 
\end{proof}

\section{A characteristic matrix function for the twisted monodromy operator} \label{sec:cm}

We next show that we can compute the eigenvalues of the twisted monodromy operator $U_h$ by computing the zeroes of a scalar-valued function. This 
translates the infinite dimensional eigenvalue problem of $U_h$ to a one dimensional zero finding problem. 
The concept of a characteristic matrix function, as introduced in \cite{KaashoekVL22}, provides the theoretical framework for this dimension reduction. 
We start this section by giving an overview of the relevant definitions and results from \cite{KaashoekVL22}, 
and then discuss how the concept of a characteristic matrix function can be applied to the twisted monodromy operator \eqref{eq:twisted mon op}. 

\medskip

In the following, we denote by $\mathcal{L}(X, X)$ the space of bounded linear operators on a Banach space $X$. We denote by $I_{X}$ the identity operator on $X$, but supress the subscript whenever the underlying space is clear. 

\begin{defn}[{\cite[Definition 5.2.1]{KaashoekVL22}}]
Let $X$ be a complex Banach space, $T: X \to X$ a bounded linear operator and $\Delta: \mathbb{C} \to \mathbb{C}^{N \times N}$ an analytic matrix-valued function. We say that $\Delta$ is a \textbf{characteristic matrix function} for $T$ if there exist analytic functions
\[ E, F: \mathbb{C} \to \mathcal{L}\left(\mathbb{C}^N \oplus X, \mathbb{C}^N \oplus X \right) \]
such that $E(z), F(z)$ are invertible operators for all $z \in \mathbb{C}$ and such that 
\begin{equation} \label{eq:conjugatie}
\begin{pmatrix}
\Delta(z) & 0 \\
0 & I_X 
\end{pmatrix} 
= F(z) \begin{pmatrix}
I_{\mathbb{C}^N} & 0 \\
0 & I-zT
\end{pmatrix} E(z)
\end{equation}
holds for all $z \in \mathbb{C}$. 
\end{defn}

Given a non-zero complex number $\mu \in \mathbb{C}$, the equality \eqref{eq:conjugatie} implies that the operator $I-\mu T$ is not invertible if and only if $\Delta(\mu)$ is not invertible, i.e. if and only if $\det \Delta(\mu) = 0$. But the operator $I- \mu T$ is not invertible if and only if $\mu^{-1}$ is in the spectrum of $T$. So if $\Delta$ is a characteristic matrix function for $T$, then $\mu^{-1}$ is in the spectrum of $T$ if and only if $\mu$ is a zero of the equation
\begin{equation} \label{eq:ce T}
\det \Delta(z) = 0
\end{equation}
and computing the roots of the equation \eqref{eq:ce T} is equivalent to computing the non-zero spectrum of the operator $T$. 

The characteristic matrix function also captures other properties of the spectrum of $T$, such as geometric and algebraic multiplicity of the non-zero eigenvalues. We have summarized the relevant results on the connection between the spectrum of $T$ and the characteristic matrix function in the following lemma, which we cite without proof from \cite{KaashoekVL22}.  

\begin{lemma}[cf. {\cite[Theorem 5.2.2]{KaashoekVL22}}] \label{lem:cm properties} Let $X$ be a complex Banach space, $T: X \to X$ a bounded linear operator and $\Delta: \mathbb{C} \to \mathbb{C}^{N \times N}$ a characteristic matrix function for $T$. Then the following statements hold: 
\begin{enumerate}
\item The non-zero spectrum of $T$ consists of isolated eigenvalues of finite algebraic multiplicity, and $\mu^{-1} \in \mathbb{C} \backslash \{ 0 \}$ is an non-zero eigenvalue of $T$ if and only if 
\[ \det \Delta(\mu) = 0. \]
\item If $\mu^{-1} \in \mathbb{C} \backslash \{ 0 \}$ is a non-zero eigenvalue of $T$, then its geometric multiplicity equals the dimension of the space
\[ \ker \Delta(\mu). \]
\item If $\mu^{-1} \in \mathbb{C} \backslash \{ 0 \}$ is a non-zero eigenvalue of $T$, then its algebraic multiplicity equals the order of $\mu$ as a root of 
\[ \det \Delta(z) =0. \]
\end{enumerate}
\end{lemma}

Characteristic matrix functions have applications in the stability analysis of periodic solutions of DDE:  
under additional conditions on the relation between the period and the delay, characteristic matrix functions can be used to determine the stability of non-symmetric periodic orbits and of discrete waves.
Suppose that the equation
\begin{equation} \label{eq:period equals delay}
\dot{x}(t) = g(x(t), x(t-\tau))
\end{equation}
has no symmetries, but does have a periodic solution $x_\ast$ with period $\tau > 0$, i.e. the period of $x_\ast$ is equal to the time delay in \eqref{eq:period equals delay}. 
If we denote by $U(t, s), \ t \geq s$, the family of solution operators associated to the linearized equation
\[ \dot{y}(t) = \partial_1 g(x_\ast(t), x_\ast(t)) y(t) + \partial_2 g(x_\ast(t), x_\ast(t)) y(t-\tau), \]
then the monodromy operator $U(\tau, 0)$ has a characteristic matrix function, as was proven in \cite{VL92} and \cite[Section 11.4]{KaashoekVL22}. The spectrum of the monodromy operator determines whether $x_\ast$ is a stable solution of \eqref{eq:period equals delay}, and hence we can determine the stability of the periodic solution $x_\ast$ by computing roots of a scalar-valued function. 
This result is particularly useful in the context of non-equivariant Pyragas control, since in system \eqref{eq:pyragas} the period of the targeted periodic solution is indeed equal to the time delay; see \cite{Fiedler20, Sieber13} for applications of this fact.  

If the function $g$ in \eqref{eq:period equals delay}
is equivariant with respect to a group $\Gamma \subseteq O(N)$, i.e. if
\begin{equation} \label{eq:equivariance dde}
g(\gamma x, \gamma y) = \gamma g(x, y) \qquad \mbox{for all } x, y \in \mathbb{R}^N \mbox{ and } \gamma \in \Gamma, 
\end{equation} 
then $\Gamma$ is a symmetry group of the solutions of \eqref{eq:period equals delay} and 
we can describe the symmetries of its periodic solutions in the same way as we have done for the ODE setting in Section \ref{sec:setting}.
If a %discrete wave solution 
period solution $x_\ast$ has a spatio-temporal symmetry of the form 
\begin{equation} \label{eq:shift equals delay}
hx_\ast(t) = x_\ast(t+\tau), 
\end{equation}
i.e. the time shift in the spatio-temporal pattern \eqref{eq:shift equals delay} is equal to the time delay in the DDE \eqref{eq:period equals delay}, then there exists a characteristic matrix for the twisted monodromy operator $U_h = h^{-1} U(\tau, 0)$. Since the spectrum of the twisted monodromy operator determines whether $x_\ast$ is a stable solution of \eqref{eq:period equals delay}, we can also in this case determine the stability of $x_\ast$ by computing roots of a scalar-valued function. 
We formally state this in the following proposition, where we also give an explicit expression for a characteristic matrix function for $U_h = h^{-1} U(\tau, 0)$. 
We cite the proposition without proof from \cite{deWolff22}; the proof builds upon the recent result \cite[Theorem 6.1.1]{KaashoekVL22}, which provides a characteristic matrix function for any operator that is the sum of a Volterra operator and a finite rank operator. 
%the proof is inspired by the arguments presented in \cite[Section 11.4]{KaashoekVL22}, but the difference is the application to the twisted monodromy operator in systems with symmetry. 

\begin{theorem}[cf. {\cite[Theorem 2.3]{deWolff22}}]
\label{thm:cm twisted}
Consider the DDE 
\begin{equation} \label{eq:dde thm 2}
\dot{x}(t) = g(x(t), x(t-\tau))
\end{equation}
with $f: \mathbb{R}^N \times \mathbb{R}^N$ a $C^2$-function and with time delay $\tau > 0$. Assume that
\begin{enumerate}
\item system \eqref{eq:dde thm 2} is equivariant with respect to a compact symmetry group $\Gamma \subseteq O(N)$, i.e. \eqref{eq:equivariance dde} holds; 
\item system \eqref{eq:dde thm 2} has a periodic solution $x_\ast$; %with minimal period $p > 0$;
\item the periodic solution %$x_\ast$ is a discrete wave and there exists 
has a spatio-temporal symmetry $h \in H$ with
\[ hx_\ast(t) = x_\ast(t+\tau). \]
\end{enumerate}
Let $U(t, s), \ t \geq s$ be the family of solution operators of the linearized DDE 
\begin{equation} \label{eq:linearized dde thm}
\dot{y}(t) = \partial_1 g(x_\ast(t), x_\ast(t-\tau))y(t) +  \partial_2 g(x_\ast(t), x_\ast(t-\tau)) y(t-\tau). 
\end{equation}
For $z \in \mathbb{C}$, let $F(t, z)$ be the fundamental solution of the ODE
\[ \dot{y}(t) = \partial_1 g(x_\ast(t), x_\ast(t-\tau))y(t) +  z \cdot \partial_2 g(x_\ast(t), x_\ast(t-\tau)) y(t) \]
with $F(0, z) = I_{\mathbb{C}^N}$. 
Then the analytic function
\[ \Delta(z) = I_{\mathbb{C}^N} - zh^{-1} F(\tau, z) \]
is a characteristic matrix function for the operator
\[ U_h = h^{-1} U(\tau, 0). \]
\end{theorem}

\noindent 
Returning to the topic of equivariant Pyragas control, 
we can write the controlled system \eqref{eq:control thm} as
\begin{equation} \label{eq:g dde}
\dot{x}(t) = g(x(t), x(t-\theta_h))
\end{equation}
with $g: \mathbb{R}^N \times \mathbb{R}^N \to \mathbb{R}^N$ given by
\begin{equation} \label{eq:g control}
g(x, y) = f(x) + k \left[x - h y \right], \qquad x, y \in \mathbb{R}^N. 
\end{equation}
For $h$ as in \eqref{eq:control thm} and fixed $j \in \mathbb{N} \cup \{0\}$, the function $g$ defined in \eqref{eq:g control} satisfies
\[ g(h^j x, h^j y) = h^j g(x, y) \qquad \mbox{for all } x, y \in \mathbb{R}^N \]
and hence the controlled system \eqref{eq:control thm} is equivariant with respect to the group generated by $h$ (cf. Section \ref{sec:equivariance}). By construction, the system \eqref{eq:control thm} has a discrete wave solution $x_\ast$ which satisfies
\begin{equation} \label{eq:shift}
hx_\ast(t) = x_\ast(t+\theta_h), 
\end{equation}
i.e. the time shift in the spatio-temporal relation \eqref{eq:shift} is equal to the time delay in \eqref{eq:g dde}. So the system \eqref{eq:control thm} satisfies the assumptions of Proposition \ref{thm:cm twisted}, 
and applying Proposition \ref{thm:cm twisted} gives a characteristic matrix function $\Delta$ for  the twisted monodromy operator. 

The expression for $\Delta$ involves the fundamental solution of the parametrized family of ODE 
\begin{equation} \label{eq:cm stap1}
\dot{y}(t) = f'(x_\ast(t)) y(t) + b \left[y(t) - z y(t) \right], \qquad z \in \mathbb{C}.
\end{equation}
Due to the fact that the control gain $b$ in \eqref{eq:cm stap1} is a \emph{scalar}, we can relatively explicitly compute the fundamental solution of \eqref{eq:cm stap1}, and hence also obtain a concrete expression of the characteristic matrix $\Delta$. We do this in the following corollary. 
 
\begin{cor} \label{cor:cm scalar control}
Consider the ODE \eqref{eq:ode} satisfying Hypothesis \ref{hyp:theorem}. For a fixed spatio-temporal symmetry $h \in H$ and scalar control gain $b \in \mathbb{R}$, 
let $U(t, s), \ t \geq s$, be the family of solution operators associated to the linear DDE \eqref{eq:linearized control}. 
Moreover, let $Y(t), \ t \in \mathbb{R}$, be the fundamental solution of the linear ODE 
\begin{equation} \label{eq:lin ode}
\dot{y}(t) = f'(x_\ast(t)) y(t) 
\end{equation}
with $Y(0) = I_{\mathbb{C}^N}$. 
Then the analytic function 
\[ \Delta(z) = I_{\mathbb{C}^N} - z\left[h^{-1} Y(\theta_h)\right] e^{b(1-z)\theta_h} \]
is a characteristic matrix function for the operator
\begin{equation} \label{eq:twisted cm}
U_h = h^{-1} U(\theta_h, 0).
\end{equation}

In particular, this means the following: 
Denote by $\mu_1, \ldots, \mu_N$ the (possibly non-distinct) eigenvalues of the matrix $h^{-1} Y(\theta_h) \in \mathbb{R}^{N \times N}$, i.e. 
\[ \{\mu_1, \ldots, \mu_N \} = \sigma(h^{-1} Y(\theta_h)) \]
and define the function
\begin{equation} \label{eq:ce control}
d(z) = \prod_{j = 1}^N \left(1-z \mu_j e^{b(1-z)\theta_h} \right). 
\end{equation}
Then $\mu^{-1} \in \mathbb{C} \backslash \{0 \}$ is an eigenvalue of the twisted monodromy operator \eqref{eq:twisted cm} if and only if $\mu$ is a root of $d(z) = 0$. 
\end{cor}
\begin{proof}
If $Y(t)$ is the fundamental solution of the ODE \eqref{eq:lin ode} with $Y(0) = I$, then
\[ F(t, z) = Y(t) e^{b(1-z)t}, \qquad z \in \mathbb{C} \]
is the fundamental solution of the ODE \eqref{eq:cm stap1} with $F(0, z) = I$. So by Proposition \ref{thm:cm twisted}, the analytic function 
\begin{align*}
\Delta(z) &= I_{\mathbb{C}^N} - h^{-1} F(\theta_h, z) \\
&= I_{\mathbb{C}_N} - z \left[h^{-1} Y(\theta_h) \right] e^{b(1-z)\theta_h}
\end{align*}
is a characteristic matrix function for the operator \eqref{eq:twisted cm}, as claimed.

\medskip

Since the multiplicative factor $e^{b(1-z)\theta_h}$ is a complex scalar, we have that 
\[ \det \Delta(z) = \prod_{j = 1}^N \left(1-z \mu_j e^{b(1-z)\theta_h} \right), \]
where $\mu_1, \ldots, \mu_N$ are the eigenvalues of $h^{-1} Y(\theta_h) \in \mathbb{R}^{N \times N}$. Therefore, Lemma \ref{lem:cm properties} implies that $\mu^{-1} \in \mathbb{C} \backslash \{0 \}$ is an eigenvalue of \eqref{eq:twisted cm} if and only $\mu$ is a root of $d(z) = 0$, with $d$ as in \eqref{eq:ce control}. 
\end{proof}

\section{Eigenvalues analysis and proof of the main theorem} \label{sec:eigenvalues}

The function $d$ defined in \eqref{eq:ce control} is the product of $N$ factors. So to find the zeroes of $d$, it suffices to find the zeroes of each of the individual factors, i.e. to find roots of equations of the form 
\begin{align} \label{eq:factor}
1 - z \mu_\ast e^{b(1-z)\theta_h} &= 0 \qquad \mbox{with } \mu_\ast \in \mathbb{C} \mbox{ and } b \in \mathbb{R} \nonumber \\
\intertext{or, by setting $b_\ast = b \theta_h$, equations of the form}
1 - z \mu_\ast e^{b_\ast(1-z)} &= 0 \qquad \mbox{with } \mu_\ast \in \mathbb{C} \mbox{ and } b_\ast \in \mathbb{R}.
\end{align}
The equation \eqref{eq:factor} was previously studied in \cite{Miyazaki11}. In fact, Lemma \ref{cor:trivial ce} and Lemma \ref{cor:unstable ce} in this section are similar to \cite[Lemma 7.4]{Miyazaki11}, but we provide new and simpler proofs. 
%but the proofs we provide here are different. %new and simpler compared to [...]. 
Our strategy here is to prove a connection between the roots of \eqref{eq:factor} and roots of the equation
\begin{equation} \label{eq:ce chap11}
0 = -z + \alpha + \beta e^{-z}
\end{equation}
with $\alpha, \beta \in \mathbb{R}$. %This connection is significant because 
Equation \eqref{eq:ce chap11} also appears in the stability analysis of equilibria of autonomous DDE; in that context, study of \eqref{eq:ce chap11} dates back at least at least to \cite{Hayes50, Bellman63}, see also \cite{HaleVL93, Diekmann95} for more recent accounts. 
%{\color{blue}and has been studied extensively in that context, see for example \cite{Hayes50, Diekmann95}. }
So we can analyze the equation \eqref{eq:factor} by first exploiting the connection with \eqref{eq:ce chap11} and then applying the results from \cite{Hayes50, Bellman63, HaleVL93, Diekmann95}. This significantly simplifies the stability analysis compared to \cite{Miyazaki11} and makes the argument more systematic. In addition, the connection between \eqref{eq:factor} and \eqref{eq:ce chap11} provides a new link between Floquet multipliers of a non-autonomous equation and eigenvalues of an autonomous equation. This connection unexpected in the context of DDE and we therefore believe it to be interesting in its own right; see also the discussion in Section \ref{sec:outlook}. 

\medskip

We first relate solutions of the equation
\begin{equation}\label{eq:factor 2}
0 = 1-ze^{\alpha} e^{\beta z}
\end{equation}
with $\alpha, \beta \in \mathbb{R}$ 
to solutions of the equation
\[ 0 = -z + \alpha + \beta e^{-z}. \]

\begin{lemma}\label{lem:exponential}
Let $\alpha, \beta \in \mathbb{R}$ and consider the analytic functions
\begin{align*}
F: \mathbb{C} \to \mathbb{C}, \qquad F(z) &= 1-ze^{\alpha} e^{\beta z} \\
G: \mathbb{C} \to \mathbb{C}, \qquad G(z) &= -z + \alpha + \beta e^{-z}. 
\end{align*}
Then $\mu \neq 0$ is a solution of $F(z) = 0$ if and only if $\mu = e^{-\lambda}$, where $\lambda$ is a solution of $G(z) = 0$. So
\begin{equation}\label{eq:set equality}
 \big \{ \mu \in \mathbb{C} \backslash \{0 \} \mid F(\mu) = 0 \big \} = \big \{ e^{-\lambda} \mid G(\lambda) = 0 \big \}.
\end{equation}
\end{lemma}
\begin{proof} For $z \in \mathbb{C}$, it holds that 
\begin{equation}\label{eq:relation}
F(e^{-z}) = 1-e^{G(z)}.
\end{equation}
So if $G(z) = 0$, then $F(e^{-z}) = 0$, and hence the map 
\begin{equation}\label{eq:map sets}
\begin{aligned}
\big \{ \lambda \in \mathbb{C} \mid G(\lambda) = 0 \big \} &\to \big \{ \mu  \in \mathbb{C} \backslash \{0 \} \mid F(\mu) = 0 \big \} \\
\lambda &\mapsto e^{-\lambda} 
\end{aligned}
\end{equation}
is well-defined. We show that the map \eqref{eq:map sets} is bijective.

\medskip
\textsc{Injective:} suppose that $\lambda, \nu$ satisfy $G(\lambda) = 0, G(\nu) = 0$ and $e^{-\lambda} = e^{-\nu}$. Then 
\begin{align*}
\lambda &=  \alpha + \beta e^{-\lambda} \\
\nu &=  \alpha + \beta e^{-\nu}
\end{align*} 
but since $e^{-\nu} = e^{-\mu}$, this implies that $\lambda = \nu$. Hence the map \eqref{eq:map sets} is injective. 

\medskip

\textsc{Surjective:} let $\mu \in \mathbb{C} \backslash \{0 \}$ be such that $F(\mu) = 0$ and let $\tilde{\lambda} \in \mathbb{C}$ be such that $e^{-\tilde{\lambda}} = \mu$. Then 
\[ 1 = e^{-\tilde{\lambda}} e^{\alpha} e^{\beta e^{-\tilde{\lambda}}} \]
and hence
\[ -\tilde{\lambda} +\alpha + \beta e^{-\tilde{\lambda}} = 2 \pi i k \]
for some $k \in \mathbb{Z}$. But then $\lambda := \tilde{\lambda} +2 \pi i k$ satisfies
\[ - \lambda + \alpha + \beta e^{-\lambda} = 0 \]
and $e^{-\lambda} = \mu$. So the map \eqref{eq:map sets} is surjective. We conclude that the map \eqref{eq:map sets} is bijective and the statement of the lemma follows. 
\end{proof}

Substituting $z = \pm i \omega, \ \omega \geq 0$ in \eqref{eq:factor} and solving for $\alpha, \beta$ gives that $\pm i \omega$ is a pair of roots of \eqref{eq:factor} if 
\begin{equation}\label{eq:imaginary ev}
\alpha = \frac{\omega \cos(\omega)}{\sin(\omega)}, \qquad \beta = -\frac{\omega}{\sin(\omega)}, \qquad \mbox{with } \omega \in (\pi k, \pi(k+1)) \mbox{ and } k \in \mathbb{N} \cup \{0 \}. 
\end{equation}
Moreover, $z = 0$ is a solution of \eqref{eq:factor} if 
\begin{equation}\label{eq:zero ev}
\alpha + \beta = 0. 
\end{equation}
For $\omega = 0$, the curve \eqref{eq:imaginary ev} intersects the line \eqref{eq:zero ev} at $(\alpha, \beta) = (1, -1)$, corresponding to $z = 0$ being a double root of \eqref{eq:factor}. Otherwise the curves defined by \eqref{eq:imaginary ev}--\eqref{eq:zero ev} do not intersect each other, and they divide the $(\alpha, \beta)$-parameter plane into regions, as is graphically depicted in Figure \ref{fig:curves}. 

A more detailed analysis, as was for example done in \cite[Chapter XI]{Diekmann95}, shows that the 
the curve
\begin{align*}
\alpha = \frac{\omega \cos(\omega)}{\sin(\omega)}, \qquad \beta = -\frac{\omega}{\sin(\omega)}, \qquad 0 \leq \omega < \pi
\end{align*}
and the half line
\[ \alpha + \beta = 0, \qquad \alpha < 1 \]
bound the \textbf{stability region} of \eqref{eq:ce chap11}, i.e the region of parameter values for which \eqref{eq:ce chap11} has no eigenvalue in the strict right half of the complex plane. 
We make this precise in the following proposition, which we take without proof from \cite[Chapter XI]{Diekmann95}.  

\begin{prop}[cf. {\cite[Chapter XI]{Diekmann95}}] \label{prop: chap 11}
Consider the equation \eqref{eq:ce chap11} with $\alpha, \beta \in \mathbb{R}$ and $z \in \mathbb{C}$. In the $(\alpha, \beta)$-plane, consider the half-line
\begin{equation} \label{eq:R}
R = \left \{ (\alpha, \beta) \mid \alpha + \beta = 0, \ \alpha < 1 \right \} 
\end{equation}
and the curve
\begin{equation} \label{eq:C}
 C = \left \{ (\alpha, \beta) \mid \alpha = \frac{\omega \cos (\omega)}{\sin(\omega)}, \  \beta = - \frac{\omega}{\sin(\omega)}, \ 0 \leq \omega < \pi \right \}.
\end{equation}
Denote by $S$ the region in the $(\alpha, \beta)$-plane bounded by $R$ and $C$ (see Figure \ref{fig:stability region}). It holds that
\begin{enumerate}
\item For $(\alpha, \beta)$ on the half-line $R$, $z = 0$ is a solution of \eqref{eq:ce chap11} and all other solutions $z$ satisfy $\re z < 0$. 
\item For $(\alpha, \beta)$ on the curve $C$, \eqref{eq:ce chap11} has two solutions on the imaginary axis and all other solutions $z$ satisfy $\re z < 0$. 
\item For $(\alpha, \beta)$ in the interior of $S$, all roots $z$ of \eqref{eq:ce chap11} satisfy $\re z < 0$. 
\item For $(\alpha, \beta) \in \mathbb{C} \backslash S$, \eqref{eq:ce chap11} has at least one root $z$ satisfying $\re z > 0$. 
\end{enumerate}
\end{prop}

\vspace{2 \baselineskip}

\begin{minipage}[t]{0.45\textwidth}
\fbox{
\includegraphics[width = 1.05\textwidth]{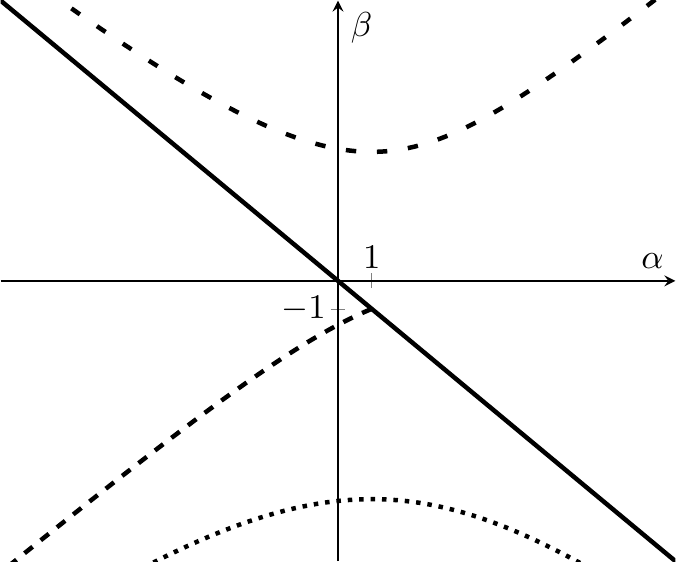}}
\captionof{figure}{Line \eqref{eq:zero ev} (solid), curve \eqref{eq:imaginary ev} with $0 \leq \omega < \pi$ (dashed), curve \eqref{eq:imaginary ev} with $\pi < \omega<2 \pi$ (loosely dashed), and curve \eqref{eq:imaginary ev} with $2 \pi < \omega < 3 \pi$ (dotted).}
\label{fig:curves}
\end{minipage} %~
\hfill
\begin{minipage}[t]{0.45\textwidth}
\fbox{
\includegraphics[width = 1.05\textwidth]{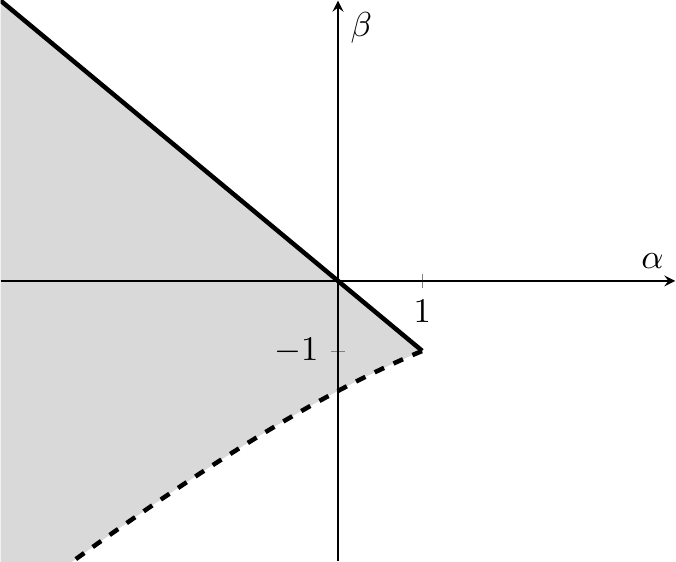}}
\captionof{figure}{The line $R$ defined in \eqref{eq:R} (solid), the curve $C$ defined in \eqref{eq:C} (dashed) and the stability region $S$ (shaded).}
\label{fig:stability region}
\end{minipage}

\vspace{2\baselineskip}

We are now ready to analyze the zeroes of the function \eqref{eq:ce control}, or, equivalently, to analyze the roots of the equation \eqref{eq:factor}. By Lemma \ref{lem:cm properties} and Corollary \ref{cor:cm scalar control}, a non-zero number $\mu^{-1} \neq 0$ is an eigenvalue of the twisted monodromy operator if and only if $\mu$ is a zero of the characteristic equation \eqref{eq:ce control}. 
%By Lemma \ref{lem:cm properties}, $\mu \neq 0$ is a zero of \eqref{eq:ce control} if and only if $\mu^{-1}$ is an eigenvalue of the monodromy operator. 
Therefore, to prove that all non-trivial eigenvalues of the twisted monodromy operator are \emph{inside} the unit circle, we have to prove that all non-trivial zeroes of \eqref{eq:ce control} are \emph{outside} the unit circle. 

%Since in Theorem \ref{result} we stabilize with a control gain $k < 0$, we also consider $k < 0$ in the analysis of the factors \eqref{eq:factor}. 	

\begin{cor}[Case $\mu_\ast = 1$] \label{cor:trivial ce}
Let $b_\ast < 0$. Then the equation
\begin{equation}\label{eq:trivial ce}
1-z e^{b_\ast(1-z)} = 0
\end{equation}
has a simple root $z = 1$; all other roots of \eqref{eq:trivial ce} lie strictly outside the unit circle. 
\end{cor}
\begin{proof}
Equation \eqref{eq:trivial ce} is of the form
\[ 1 - ze^\alpha e^{\beta z} = 0 \]
with 
\[ \alpha(b_\ast) = b_\ast, \qquad \beta(b_\ast) = -b_\ast. \]
For $b_\ast < 0$, the point $(\alpha(b_\ast), \beta(b_\ast))$ lies on the line $R$ as defined in \eqref{eq:R}. Therefore, the equation
\[ -z + \alpha(b_\ast) + \beta(b_\ast) e^{-z} = 0\]
has a solution $z = 0$ %, which is algebraically simple, 
and all other solutions lie in the strict left half plane. Lemma \ref{lem:exponential} then implies that equation \eqref{eq:trivial ce} has a solution $z = 1$ %, which is algebraically simple, 
and that all other roots of \eqref{eq:trivial ce} lie strictly outside the unit circle. 

To prove that $z = 1$ is a simple solution of \eqref{eq:trivial ce} for $b_\ast < 0$, we compute
\begin{align*}
\left. \frac{d}{dz}\right|_{z=1} \left( 1-z e^{b_\ast(1-z)} \right) &= \left. - e^{b_\ast(1-z)} + b_\ast z e^{b_\ast(1-z)} \right|_{z = 1}  \\ 
&= -1 + b_\ast \neq 0
\end{align*}
for $b_\ast < 0$. 
\end{proof}

\begin{cor}[Case $\mu_\ast < -1$] \label{cor:unstable ce} Consider the equation
\begin{equation}\label{eq:unstable ce}
1 - \mu_\ast z e^{b_\ast(1-z)} = 0
\end{equation}
and for fixed $\mu_\ast$ define the set 
\begin{equation} 
I(\mu_\ast) := \{b_\ast \in \mathbb{R} \mid \mbox{all solutions of \eqref{eq:unstable ce} lie strictly outside the  unit circle} \}. \label{eq:set I}
\end{equation}
Then the following two statements hold: 
\begin{enumerate}
\item If $-e^2 < \mu_\ast < -1$, the set $I(\mu_\ast)$
is a \emph{non-empty} open interval with $I(\mu_\ast) \subseteq (-\infty, 0)$. 
\item If $\mu_{\ast, 1}, \ \mu_{\ast, 2}$ are two real numbers satisfying  $-e^2 < \mu_{\ast, 1} < \mu_{\ast, 2} < -1$, 
then it holds that 
\begin{equation} \label{eq:nesting}
I(\mu_{\ast, 1}) \subseteq I(\mu_{\ast, 2}).
\end{equation}
\end{enumerate}
\end{cor}
\begin{proof}
To prove the first statement of the lemma, let $\mu_\ast$ be a real number with $-e^2 < \mu_\ast < -1$. 
A complex number $\mu \in \mathbb{C}$ is a solution of \eqref{eq:unstable ce} if and only if $\nu : = - \mu$ is a solution of 
\begin{equation}\label{eq:ce nu}
1 - z [-\mu_\ast] e^{b_\ast (1 + z)} = 0 . 
\end{equation}
Hence all solutions of \eqref{eq:unstable ce} lie strictly outside the unit circle if and only if all solutions of \eqref{eq:ce nu} lie strictly outside the unit circle. 
%If we now $\mu_\ast$ satisfies $- e^2 <\mu_\ast <-1$, 
Since by assumption $-e^2 < \mu_\ast < -1$, we can write $- \mu_{\ast} = e^{\lambda_\ast}$ with $0 < \lambda_\ast < 2$. So \eqref{eq:ce nu} is of the form 
\[
1-ze^{\alpha} e^{\beta z} = 0
\]
with 
\begin{equation}\label{eq:path}
\alpha(b_\ast) = \lambda_\ast + b_\ast, \qquad \beta(b_\ast) = b_\ast. 
\end{equation}
By Lemma \ref{lem:exponential}, the solutions of \eqref{eq:ce nu} lie strictly outside the unit circle %(and hence all solutions of \eqref{eq:unstable ce} lie strictly outside the unit circle) 
if and only if all solutions of 
\begin{equation} \label{eq:ce alpha}
-z + \alpha(b_\ast) + \beta(b_\ast)e^{-z} = 0
\end{equation}
lie in the strict left half of the complex plane. But by Proposition \ref{prop: chap 11}, all roots of \eqref{eq:ce alpha} lie in the strict left half of the complex plane if and only if the point \eqref{eq:path} lies inside the region $S$ (as defined in Proposition \ref{prop: chap 11}). So to determine the elements of the set $I(\mu_\ast)$, we have to find the values of $b_\ast$ for which the point \eqref{eq:path} lies in the set $S$ as defined in Proposition \ref{prop: chap 11}.  

If we view \eqref{eq:path} as a path parametrized by $b_\ast$, then for $b_\ast = -\lambda_\ast/2$ this path crosses the line $R = \{(\alpha, \beta) \mid \alpha + \beta = 0, \ \alpha < 1 \}$ in the point
\[ (\alpha, \beta) = \left(\frac{\lambda_\ast}{2}, - \frac{\lambda_\ast}{2}\right), \]
and there exists 
%We then see from Figure \ref{fig:stability region} that there exists 
a $\bar{b} < - \lambda_\ast/2$ such that for $b_\ast$ in the interval
\[\left(\bar{b},  - \frac{\lambda_\ast}{2}\right) \]
the path \eqref{eq:path} lies inside the set $S$ as defined in Proposition \ref{prop: chap 11}; see Figure \ref{fig:path}. So the set $I(\mu_\ast)$ defined in \eqref{eq:set I} is of the form $I(\mu_\ast) = (\bar{b}, -\lambda_\ast/2)$, which proves the first statement of the lemma.  

\medskip

To prove the second statement of the lemma, we let $\mu_\ast^1, \mu_\ast^2$ be two real numbers with \[-e^2 < \mu_{\ast, 1} < \mu_{\ast,2} < -1 \] (so $\mu_{\ast, 1}$ has larger modulus than $\mu_{\ast, 2}$) and write 
\[ - \mu_{\ast, 1} = e^{\lambda_{\ast, 1}}, \qquad - \mu_{\ast, 2} = e^{\lambda_{\ast,2}} \]
with $\lambda_{\ast, 1} > \lambda_{\ast, 2}$. Now let $b_\ast \in I(\mu_{\ast, 1})$, which means that the point 
\begin{equation} \label{eq:point1}
(\alpha, \beta) = (\lambda_{\ast, 1} + b_\ast, b_\ast) %\alpha_1(b_\ast) = \lambda_{\ast, 1} + b_\ast, \qquad \beta_1(b_\ast) = b_\ast,
\end{equation}
lies inside the region $S$. Since $\lambda_{\ast, 2} < \lambda_{\ast, 1}$, we can obtain the point 
\begin{equation} \label{eq:point2}
(\alpha, \beta) = (\lambda_{\ast, 2} + b_\ast, b_\ast) %\alpha_2(b_\ast) = \lambda_{\ast, 2} + b_\ast, \qquad \beta_2(b_\ast) = b_\ast, \qquad b_\ast \leq 0
\end{equation}
by translating the point \eqref{eq:point1} horizontally to the left. But we then see from Figure \ref{fig:two paths} that if \eqref{eq:point1} lies inside the region $S$, then \eqref{eq:point2} lies inside the region $S$ as well, and hence $b_\ast \in I(\mu_{\ast, 2})$. So it holds that $I(\mu_{\ast, 1}) \subseteq I(\mu_{\ast, 2})$, as claimed. 
\end{proof}

\vspace{2 \baselineskip}

\begin{minipage}[t]{0.45\textwidth}
\fbox{
\includegraphics[width = 1.05\textwidth]{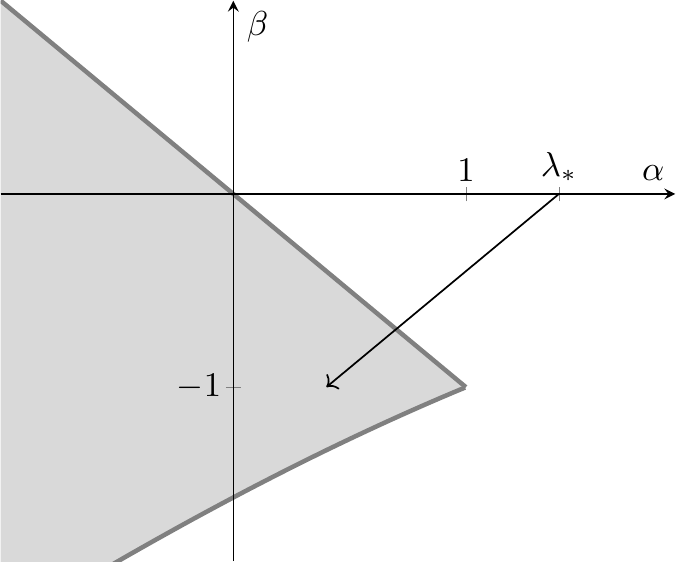}}
\captionof{figure}{The stability region $S$ (grey) together with the path defined by \eqref{eq:path} for $b_\ast \leq 0$ (black). The path defined by \eqref{eq:path} enters the region $S$ at the point $(-\lambda_\ast/2, \lambda_\ast/2)$.}
\label{fig:path}
\end{minipage} %~
\hfill
\begin{minipage}[t]{0.45\textwidth}
\fbox{
\includegraphics[width = 1.05\textwidth]{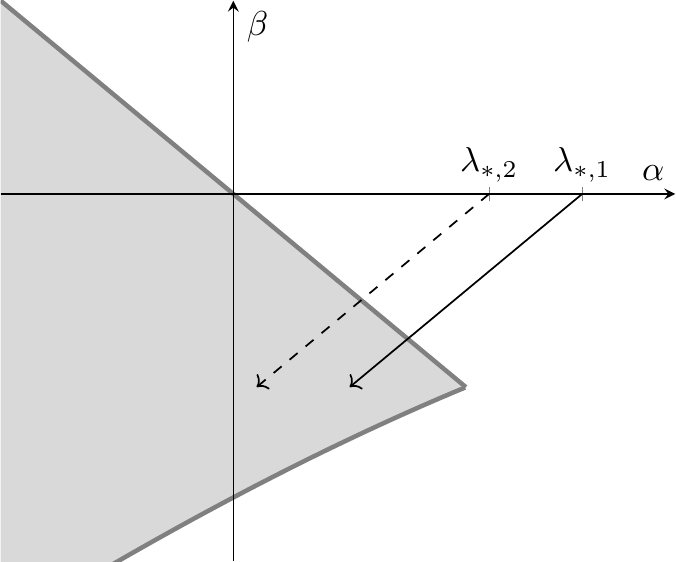}}
\captionof{figure}{The stability region $S$ (grey) together with the paths defined by \eqref{eq:point1} for $b_\ast \leq 0$ (black solid) and \eqref{eq:point2} with $b_\ast \leq 0$ (black dashed). The path \eqref{eq:point2} is a left translate of the path \eqref{eq:point1}; if for a given value of $b_\ast$ the point \eqref{eq:point1} lies inside the region $S$, then for that value of $b_\ast$ the point \eqref{eq:point2} lies inside the region $S$ as well. }
\label{fig:two paths}
\end{minipage}

\vspace{2\baselineskip}

In the case where $\mu_\ast \in \mathbb{C} \backslash \mathbb{R}$, the equation \eqref{eq:factor} can be only brought into the form $1-ze^{\alpha}e^{\beta z} =0$ by choosing $\alpha$ to be complex. So in this case, Proposition \ref{prop: chap 11} does not give information on the roots of \eqref{eq:factor}. However, for $\left| \mu_\ast \right| < 1$, we can analyze the roots by direct estimates. 

\begin{cor}[Case $\left|\mu_\ast \right| < 1$, cf. statement (v) in Lemma 7.4 in \cite{Miyazaki11}] \label{cor:stable ce}
Let $\mu_\ast \in \mathbb{C}, \ \left| \mu_\ast \right| < 1$ and $b_\ast < 0$. Then all solutions of 
\begin{equation}\label{eq:stable ce}
1 - \mu_\ast z e^{b_\ast(1-z)} = 0
\end{equation}
lie strictly outside the unit circle. 
\end{cor}
\begin{proof}
Let $\mu \in \mathbb{C}$ be a solution of \eqref{eq:stable ce}, then 
\[ 1 = \left| \mu_\ast \right| \left| \mu \right|  e^{b_\ast}  \left| e^{-b_\ast \mu} \right|. \]
Because $b_\ast < 0$, it holds that $\left|- b_\ast \mu \right| = -b_\ast \left| \mu \right|$ and hence $ \left| e^{-b_\ast \mu} \right| \leq e^{-b_\ast \left|  \mu \right|}$. So $\mu$ satisfies the estimate
\begin{equation}\label{eq:estimate stable}
1 \leq  \left| \mu_\ast \right| \left| \mu \right| e^{b_\ast} e^{-b_\ast \left| \mu \right|}. 
\end{equation}
Now assume by contradiction that $\left| \mu \right| \leq 1$. Since $b_\ast < 0$, we can estimate the right hand side of \eqref{eq:estimate stable} as
\begin{align*}
\left| \mu_\ast \right| \left| \mu \right| e^{b_\ast} e^{-b_\ast \left| \mu \right|}&\leq \left| \mu_\ast \right| e^{b_\ast} e^{-b_\ast} \\ 
& \leq \left| \mu_\ast \right| < 1.
\end{align*}
But this contradicts the estimate \eqref{eq:estimate stable}. We conclude that if $\mu$ is a solution of \eqref{eq:unstable ce}, then $\left| \mu \right| \geq 1$. 
\end{proof}

Having assembled all the necessary ingredients, we now prove Theorem \ref{thm:main result}. 

\begin{proof}[Proof of Theorem \ref{thm:main result}]

Let $Y(t) \in \mathbb{R}^{N \times N}$ be the fundamental solution of the ODE 
\[ \dot{y}(t) = f'(x_\ast(t)) y(t) \]
with $Y(0) = I$. We denote by $\mu_1, \ldots, \mu_N$ the (possibly non-distinct) eigenvalues of the matrix $h^{-1} Y(\theta_h) \in \mathbb{R}^{N \times N}$, i.e. 
\[ \{ \mu_1, \ldots, \mu_N \} = \sigma \left(h^{-1} Y(\theta_h)\right). \]
Moreover, we let $U(t, s), \ t \geq s$, be the family of solution operators associated to the DDE 
\[ \dot{y}(t) = f'(x_\ast(t)) y(t) + b \left[y(t) - hy(t-\theta_h)\right] \]
and we let 
\[ U_h = h^{-1} U(\theta_h, 0) \]
be the twisted monodromy operator. 

By Proposition \ref{prop:spectral stability}, $x_\ast$ is stable solution of 
\begin{equation} \label{eq:control repeat}
\dot{x}(t) = f(x(t)) + b \left[x(t) - hx(t-\theta_h)\right]
\end{equation}
if the trivial eigenvalue $1 \in \sigma_{pt}(U_h)$ is algebraically simple and all other eigenvalues of $U_h$ lie strictly inside the unit circle. By Corollary \ref{cor:cm scalar control}, a non-zero complex number $\mu^{-1}$ is an eigenvalue of $U_h$ if and only if $\mu$ is a zero of the function 
\begin{equation} \label{eq:cf repeat}
d(z) : = \prod_{j = 1}^N\left(1-z \mu_j e^{b(1-z)\theta_h}\right), 
\end{equation}
and the trivial eigenvalue $1 \in \sigma(U_h)$ is algebraically simple if $z = 1$ is a simple zero of \eqref{eq:cf repeat}. 
So if there exists an open interval $I \subseteq \mathbb{R}$ such that for $b \in I$, $z =1$ is a simple zero of \eqref{eq:cf repeat} and all other zeroes are strictly \emph{outside} the unit circle, then the statement of the theorem follows.  

\medskip

We first consider the case in which $h^{-1} Y(\theta_h)$ has at least one eigenvalue strictly outside the unit circle. Denote by $d$ the number of eigenvalues of $h^{-1} Y(\theta_h)$ strictly outside the circle; then we can relabel the eigenvalues $\mu_1, \ldots, \mu_N$ in such a way that 
\begin{enumerate}
\item the eigenvalue $\mu_1$ has largest modulus, i.e. $\left| \mu_j \right| \leq \left| \mu_1 \right|$ for $2 \leq j \leq N$;
\item the eigenvalues $\mu_1, \ldots, \mu_d$ lie strictly outside the unit circle; 
\item the eigenvalues $\mu_{d+1}, \ldots, \mu_{N-1}$ lie strictly inside the unit circle;
\item $\mu_{N} = 1$. 
\end{enumerate}
%Denote by $\mu_\ast$ the eigenvalue of $h^{-1} Y(\theta_h)$ with largest modulus, i.e. if $\mu_i \in \sigma(h^{-1} Y(\theta_h))$ is another eigenvalue of $h^{-1} Y(\theta_h)$, then $\left| \mu_i \right| \leq \left| \mu_\ast \right|$. Necessarily the eigenvalue $\mu_\ast$ lies strictly outside the unit circle, and the second assumption in the theorem implies that $-e^2 < \mu_\ast < -1$. 
Next we consider the equation 
\begin{equation} \label{eq:ce repeat}
1 - \mu_1 z e^{b(1-z)\theta_h} = 0
%1 - \mu_\ast z e^{b(1-z)\theta_h} = 0
\end{equation}
and define the set 
\begin{equation} 
I :=  \{b \in \mathbb{R} \mid \mbox{all solutions of \eqref{eq:ce repeat} lie strictly outside the unit circle} \}. \label{eq:defn I}
\end{equation}
So with the notation of Corollary \ref{cor:unstable ce}, and with $I(\mu_1)$ as defined in \eqref{eq:set I} for $\mu_\ast =\mu_1$, we have that 
\[ I = \{b \in \mathbb{R} \mid b_\ast: = b \theta_h \in I(\mu_1) \} \]
and Corollary \ref{cor:unstable ce} implies that $I$ is a nonempty open interval with $I \subseteq (-\infty, 0)$. We now claim that for $b \in I$, the zero $z = 1$ of \eqref{eq:cf repeat} is algebraically simple, and all other zeroes of \eqref{eq:cf repeat} lie strictly outside the unit circle. 

If $1 \leq j \leq d$, i.e. if $\mu_j$ lies strictly outside the unit circle, then 
%If $\mu_j \in \sigma(h^{-1} Y(\theta_h))$ is any eigenvalue $h^{-1} Y(\theta_h)$ outside the unit circle, then 
the second assumption of Theorem \ref{thm:main result} implies that $-e^2 < \mu_j < -1$. Moreover, since $\left| \mu_1 \right| \geq \left| \mu_j \right|$, it holds that $-e^2 < \mu_1 \leq \mu_j < -1$. So in particular, the second statement of Corollary \ref{cor:unstable ce} implies that if $b \in I$ and $1 \leq j \leq d$, then all solutions of 
\[ 1-z \mu_j e^{b(1-z)\theta_h} = 0 \]
lie strictly outside the unit circle. 

If $b \in I$, then in particular $b < 0$ and also $b_\ast : = b \theta_h < 0$. So if $d+1 \leq j \leq N-1$, i.e. if $\mu_j$ lies strictly inside the unit circle, then Corollary \ref{cor:stable ce} implies that all solutions of 
\[ 1 - z \mu_j e^{b(1-z)\theta_h} = 0 \]
lie strictly outside the unit circle. Moreover, if $j = N$, i.e. for $\mu_N = 1$, Corollary \ref{cor:trivial ce} implies that
%If $b \in I$, then in particular $b < 0$ and also $b_\ast : = b \theta_h < 0$. So Corollary \ref{cor:trivial ce} implies that the equation 
\[ 1- z e^{b(1-z)\theta_h} = 0 \]
has a solution $z = 1$, which is simple, and all its other solutions lie strictly outside the unit circle. 
%Moreover, if $\mu_j$ is an eigenvalue of $h^{-1} Y(\theta_h)$ strictly inside the unit circle, then Corollary \ref{cor:stable ce} implies that all solutions of 
%\[ 1 - z \mu_j e^{b(1-z)\theta_h} = 0 \]
%lie strictly inside the unit circle. 

%So in the case where $h^{-1} Y(\theta_h)$ has an eigenvalue strictly outside the unit circle, and with $I$ defined as in \eqref{eq:defn I}, 
We conclude that if $b \in I$, the zero $z = 1$ of \eqref{eq:cf repeat} is simple, and all other zeros of \eqref{eq:cf repeat} lie strictly outside the unit circle. In this case the statement of the theorem follows. 

\medskip

In the case where the matrix $h^{-1} Y(\theta_h)$ does not have a eigenvalue strictly outside the unit circle, the assumptions of the theorem imply that $h^{-1} Y(\theta_h)$ has an algebraically simple eigenvalue $1$ and all other eigenvalues lie strictly inside the unit circle (this in particular means that $x_\ast$ is a stable solution of the ODE \eqref{eq:ode}). In this case, Corollaries \ref{cor:trivial ce} and \ref{cor:stable ce} imply that for any $b < 0$, the zero $z = 1$ of \eqref{eq:cf repeat} is algebraically simple and all other zeroes lie strictly outside the unit circle. So also in this case the statement of the theorem follows.
\end{proof}

\section{Discussion and outlook} \label{sec:outlook}

%A running theme in this article is that we do not only use symmetry in the implementation of the control scheme, but also actively use the symmetry in the stability analysis. 
A key ingredient in the proof of Theorem \ref{thm:main result} is the introduction of the \emph{twisted monodromy operator} (cf. equation \eqref{eq:twisted mon op}), which can be viewed as an adaptation of the monodromy operator to equivariant settings. The introduction of this twisted monodromy operator is crucial for two reasons.

Firstly, the properties of the twisted monodromy operator, and \emph{not} the properties of monodromy operator, determine whether stabilization via equivariant Pyragas control is possible. 
%As point in case, 
We see this in the statement of %the formulation of 
Theorem \ref{thm:main result}, %provides
which provides
clear conditions on the location of the spectrum of the twisted monodromy operator (namely on the negative real axis), while the location of the spectrum of the monodromy operator can either be on the positive or negative real axis. 
As another point in case, \cite{deWolff21} contains an invariance principle for non-equivariant Pyragas control in terms of eigenvalues of the monodromy operator; this result (and the consequences thereof)
apply verbatim to equivariant Pyragas control when we replace the monodromy operator by the twisted monodromy operator (see also \cite[Chapter 10]{proefschrift} for more precise statements). 

Secondly, Theorem \ref{thm:cm twisted} provides a $N$-dimensional characteristic matrix function for the twisted monodromy operator. However, we cannot so easily find a $N$-dimensional characteristic matrix function for the monodromy operator. %This is essentially because the time delay of the controlled system equals the time step of the spatio-temporal symmetry, whereas the full period is larger than the time delay. 
This is because for equivariant Pyragas control, the time delay %of the controlled system 
equals the time step in the spatio-temporal symmetry, but the time delay is a fraction of the full period. The existence of a characteristic matrix function for the twisted monodromy operator implies that we can compute its eigenvalues as roots of a scalar valued equation, which is advantageous from a computational point of view. 

\medskip

Lemma \ref{lem:exponential} shows that eigenvalues of the twisted monodromy operator of \eqref{eq:linearized control} are exponentially related to roots of equations of the form
\begin{equation} \label{eq:ce outro}
z = \alpha + \beta e^{-z}
\end{equation}
with $\alpha, \beta \in \mathbb{R}$. Equation \eqref{eq:ce outro} is in turn related to the linear, autonomous DDE
\begin{equation} \label{eq:dde outro}
\dot{y}(t) = \alpha y(t) + \beta y(t-1) \qquad \mbox{with } y(t) \in \mathbb{R},
\end{equation}
because roots of \eqref{eq:ce outro} are exactly the eigenvalues of the generator of the semiflow to \eqref{eq:dde outro} \cite[Chapter IV]{Diekmann95}. 
So Lemma \ref{lem:exponential} shows that eigenvalues of the twisted monodromy operator of \eqref{eq:linearized control} are exponentially related to eigenvalues associated with a linear, autonomous DDE of the form \eqref{eq:dde outro}. 

This is reminiscent of the situation in ODE, 
where there is a time-periodic transformation that transforms the periodic ODE
\begin{equation} \label{eq:ode time}
\dot{y}(t) = A(t) y(t), \qquad A(t+1) = A(t) \in \mathbb{R}^{N \times N}
\end{equation}
to an autonomous ODE of the form 
\begin{equation} \label{eq:ode constant}
\dot{y}(t) = B y(t), \qquad B \in \mathbb{R}^{N \times N};
\end{equation}
through this transformation, 
%and hence 
eigenvalues of the monodromy operator of \eqref{eq:ode time} are exponentially related to eigenvalues of the generator of \eqref{eq:ode constant}. The crucial difference with DDE, however, is that there is in general \emph{no} time-periodic transformation that transforms the DDE \eqref{eq:linearized control} to an autonomous DDE of the form \eqref{eq:dde outro}; cf \cite[Chapter XIII]{Diekmann95}.  
So it surprising that, although there is no relation between the solution operators of \eqref{eq:linearized control} and \eqref{eq:dde outro}, there \emph{is} a relation between their spectra.

It is natural to ask whether this is true for a larger class of equations, e.g. whether the spectrum associated with a time-periodic DDE of the form
\begin{equation} \label{eq:point delay}
 \dot{y}(t) = A(t) y(t) + B(t) y(t-1), \qquad \mbox{with } A(t+1) = A(t),  \ B(t+1) = B(t) \mbox{ and }  A(t), \ B(t) \in \mathbb{R}^{N \times N}
\end{equation}
%(so with period equal to the delay) 
%the spectrum 
is related to the spectrum of an autonomous equation, possibly with a distributed delay. Since the eigenvalue problem of autonomous DDE is relatively well understood (see e.g. \cite{Yanchuk22} for recent developments in this topic), such a correspondence would contribute to a further understanding of the eigenvalue problem associated to \eqref{eq:point delay}, and in particular to a further analytical understanding of Pyragas control. 
%An answer to this question would be especially relevant in the context of (equivariant) Pyragas control, where such a correspondence could contribute to a further analytical understanding of the controlled problem.  

\bibliographystyle{alpha}
\bibliography{waves} 

\end{document}